\documentclass{article}
\usepackage{amsmath,amsthm,amsfonts,color,todonotes}
\newtheorem{proposition}{Proposition}
\DeclareMathOperator*{\argmax}{arg\,max}
\usepackage{color}

\newtheorem{lemma}{Lemma}
\theoremstyle{definition}
\newtheorem{definition}{Definition} 

\title{The entry and exit game in the electricity markets: a 
  mean-field game approach}
\author{Ren\'e A\"id  \\  Universit\'e Paris Dauphine-PSL \and
  Roxana Dumitrescu \\ King's College London \and Peter Tankov \\
  ENSAE, Institut Polytechnique de Paris}
\date{}
\begin{document}
\maketitle
\begin{abstract}
We develop a model for the industry dynamics in the electricity
market, based on mean-field games of optimal stopping. In our model,
there are two types of agents: the renewable producers and the
conventional producers. The renewable producers choose the optimal
moment to build new renewable plants, and the conventional producers
choose the optimal moment to exit the market. The agents interact
through the market price, determined by matching the
aggregate supply
of the two types of producers with an exogenous demand function. Using
a relaxed formulation of optimal stopping mean-field games, we prove the existence
of a Nash equilibrium and the uniqueness of the equilibrium price
process. An empirical example, inspired by the UK electricity market is
presented.  The example shows that while renewable subsidies clearly lead to higher
renewable penetration, this may entail a cost to the consumer in terms
of higher peakload prices. In order to avoid rising prices, the
renewable subsidies must be combined with
mechanisms ensuring that sufficient conventional capacity remains in
place to meet the energy demand during peak periods. 
\end{abstract}
\section{Introduction}
The world electricity sector is undergoing a major transition. The
large-scale deployment of renewable energy fueled by tax rebates, subsidies, and
feed-in tariffs in the last 20 years has profoundly
changed the market landscape: instead of large integrated utilities, a
considerable fraction of electricity is now generated by small and
medium-sized renewable
producers. Renewable generation is already affecting prices in many
countries, and its role is bound to increase.  According to
Internatonal Energy Agency \cite{IAE2017}, to limit
the global temperature increase to 1.75$^\circ$C by 2100 (Paris Agreement
range midpoint), the energy sector must reach carbon neutrality by
2060. This objective is achievable only through massive deployment of
renewable electricity. Such levels of renewable penetration may not be
compatible with the present structure of electricity markets, networks
and incentives.  In particular, the near-zero marginal cost of
electricity from renewable sources pushes down baseload wholesale electricity
prices, eroding the profits of the  conventional producers vital for
system stability. As a result, baseload conventional producers leave
the market\footnote{According to Reuters, the US coal electricity
  generation industry has been in steep decline for a decade due to
  competition from cheap and abundant gas and subsidized solar and
  wind energy, and 39,000 MW of coal-fired generation capacity was
  shut since 2017, see \texttt{\tiny{https://www.reuters.com/article/us-usa-coal-decline-graphic/\\u-s-coal-fired-power-plants-closing-fast-despite-trumps-pledge-of-support-for-industry-idUSKBN1ZC15A}}}, and the peak demand has to be met to a larger extent by
peakload plants with much higher generation costs. Paradoxically, the
increased renewable penetration may therefore lead to higher peak
prices and increased overall electricity procurement costs \cite{forbes}.

The goal of this paper is to develop a game theoretical model to
understand the dynamics of electricity markets under large scale
renewable penetration. More precisely, we use the  setting  of mean-field games (MFG) to
describe the evolution of the future electricity markets under
different incentive schemes, to understand
the effect of these policy decisions on the entry and exit of the market
players and the evolution of renewable penetration and electricity
prices. 

We consider a stylized model with two classes of agents: (i) The
intermittent (wind) producers, who generate electricity with a
stochastic capacity factor at zero marginal cost. The renewable
producers aim to determine the optimal moment to enter the market, by
paying a sunk cost. %The running revenue of the renewable agents after
%entering the market is determined by the market price and any
%subsidies they may receive. 
(ii) The conventional (gas) producers with
a fixed capacity, but a random running cost (depending in particular
on the fuel cost and the CO2 emission cost). They aim to determine the
optimal moment to exit the market. %The revenue of the conventional producers
%before exiting the market is determined by the market price. 
The two
types of agents interact through the market price, which is deduced
from the total renewable production and the total conventional
capacity through a merit order mechanism, using an exogeneously
specified deterministic demand function.

The theory of relaxed
solutions of optimal stopping MFG, developped in \cite{bouveret2018mean}, is used to determine the dynamic equilibrium trajectory for the
baseload and peakload prices and for the conventional and renewable
installed capacity. { We prove the existence
of a Nash equilibrium and the uniqueness of the equilibrium price
process. Our proofs are based on technical tools specific to this
particular problem, which are not entirely covered by
\cite{bouveret2018mean}: in particular, the price functional deduced
from the merit order mechanism is highly irregular and requires
special treatment.}

A
numerical illustration, inspired by the UK electricity market is
presented, and allows to conclude that while renewable subsidies clearly lead to higher
renewable penetration, this may entail a cost to the consumer in terms
of higher peakload prices. In order to avoid rising prices, the
renewable subsidies must be combined with market or off-market
mechanisms ensuring that sufficient conventional capacity remains in
place to meet the energy demand during peak periods. 

The paper is structured as follows. In the remaining part of the
introduction we review the relevant literautre. Section
\ref{model.sec} presents our model of the electricity market. The
relaxed solution approach and the main theoretical results are exposed
in Section \ref{relaxed.sec}. {The numerical computation of
the MFG solution is presented in Section \ref{numerics.sec} and a concrete
example is discussed in Section \ref{illustration.sec}. Finally, in Appendix, we provide the detailed proof of a technical result used in the paper.}
\paragraph{Literature review} 
Several papers consider the effect of increased renewable penetration
on market prices of electricity both in the long-term \cite{CCZ15} and
short-term \cite{KP17} setting. Recent papers also study the interactions
between energy and capacity markets in the presence of market power,
that is, the ability of producers to willingly influence prices in a
way favorable for them, and the role of renewables in such
interactions. Schwenen \cite{Sch15} provides empirical evidence of market power in
the New York capacity market. Fabra \cite{Fab18} uses a microeconomic
model to study the effects of market power in the capacity market on
the performance of energy markets. Benatia \cite{Ben18} studies strategic bidding in New York’s energy market.

Many authors have analyzed possible market evolutions and alternative
designs to ensure system reliability and economic viability. Henriot
and Glachant \cite{HG13} discuss alternative schemes for market
integration of renewable generators, Levin and Botterud \cite{LB15}
analyze and compare different market designs to ensure generator
revenue sufficiency and Rious et al.~\cite{RPR15} study the market design to encourage the development of demand response and improve system flexibility. A particularly important innovation has been the introduction of capacity markets, surveyed in \cite{BLB18}. The above papers mainly analyze the structure of existing markets and make proposals for improvement but do not use models to simulate the evolution of future markets under alternative design proposals. This line of research has been pursued by by some authors using computational agent-based modeling (see e.g. \cite{BMR17} for an application to the capacity market).

The game theory has been applied to the modeling of entry/exit decisions of agents in electricity markets in \cite{TGKM08}, where a system of two producers is considered, and the price is not affected by the agents' decisions. A considerable literature is devoted to capacity expansion games in energy markets, see e.g., \cite{ALL17}. A related strand of literature uses computational agent-based models to understand the dynamics of wholesale electricity markets, see \cite{WV08} for a review. These models allow for heterogeneous agents and a precise description of the market structure, but are computationally very intensive and do not provide any insight about the model (uniqueness of the equilibrium, robustness etc.) beyond what can be recovered from a simulated trajectory.

The machinery of MFG appears to be a promising compromise between the
complexity of computational agent-based models and the tractability of
fully analytic approaches. MFG, introduced in \cite{LL07} are
stochastic games with a large number of identical agents and symmetric
interactions where each agent interacts with the average density of
the other agents (the mean field) rather than with each individual
agent. This simplifies the problem, leading to explicit solutions or
efficient numerical methods for computing the equilibrium dynamics. In
the recent years MFG have been successfully used to model specific
sectors of electricity markets, such as price formation \cite{gomes2020mean}, electric vehicles \cite{CPTD12}, demand dispatch \cite{BB14} and storage \cite{AMT18}.
%While the assumption of identical agents and symmetrical interactions
%is too restrictive to describe electricity markets with heterogeneous
%agents, recent developments in the MFG literature such as MFG with
%common noise \cite{CDL16} and MFG with major and minor players
%\cite{CZ16} have opened the way to more realistic models with, for
%example, explicit modeling of a large historical producer having
%market power.
An important recent development is the introduction of
MFG of optimal stopping / obstacle mean-field games \cite{B18,bouveret2018mean,gomes2015obstacle}, which can
describe technology switches and entry/exit decisions of players. In
this paper, we follow the relaxed solution approach of
\cite{bouveret2018mean} (see also \cite{gomes2014mean} for a related
notion of relaxed solution), and extend
the theoretical results of that paper to allow for two classes of
agents and enable the agents to interact through the market price,
which is an irregular functional not covered by the assumptions in
\cite{bouveret2018mean}. 

%Market design and subsidies. Wholesale electricity markets have two
%roles: ensuring short-term reliability through merit-order dispatch,
%and defining the financial incentives guaranteeing long-term economic
%viability of the system. Massive deployment of intermittent renewable
%energies modifies the existing market structure in many ways: average
%prices are reduced because of near zero marginal costs of renewable
%generation, affecting the revenue streams of other generators; prices
%become more volatile, and the need for system flexibility is
%increased. 

% Takashima et
% al. \cite{takashima2008entry}, consider the system of two potential energy
% producers who determine the optimal moment to enter the electricity
% market. In this article, the interaction is based purely on demand,
% and the price level, though random, is not affected by the decisions
% of the agents. A large strand of literature uses computational
% agent-based models to gain insight into the dynamic behavior of
% wholesale electricity markets, see \cite{weidlich2008critical} for a
% review. 
\section{The model}
\label{model.sec}
We fix a terminal time horizon $T$ and consider a model of electricity
market with exogeneous demand and producing agents of two different
types. In this section, we describe the model for a finite number of
agents, before passing to the MFG limit in the following ones. 
\paragraph{Conventional producers} The maximum capacity of all conventional
  producers is assumed to be identical and fixed; the plants can
  operate at zero, full or partial capacity.   Each conventional producer has a marginal cost function 
$$
C^i_t: [0,1]\mapsto \mathbb R,
$$
in other words, if the plant is operating at a fraction $\xi$ of its
total capacity, $C^i_t(\xi)$ corresponds to the unit cost of producing
an additional infinitesimal amount of energy. We assume
that 
$$
C^i_t(\xi) = C^i_t + c(\xi),
$$
where $C^i_t$ is the \emph{baseline cost} process and $c:\mathbb R_+
\mapsto [0,1]$ is a deterministic strictly increasing smooth function
with $c(0) = 0$. 
This increasing property is justified by the need to start up
additional less efficient units, recall {additional} employees
etc. Consequently, for a given price level $p$, the producer offers a
fraction $F(p-C^i_t)$ of its total capacity, where $F$ is the inverse
mapping of $c$. $F$ is also a smooth increasing function, and it
satisfies $F(y)=0$ for $y\leq0$, $F(y)>0$ for $y>0$ and $F(y)=1$ for
$y\geq c(1)$. In other words, the producer
does not offer any capacity if the price is less than the production
cost, and as the price increases beyond the production cost, the
producer gradually offers more capacity, until reaching the full
capacity at a sufficiently high price level.

The baseline cost is assumed to
follow the CIR process:
\begin{align}
dC^i_t = k(\theta- C^i_t) dt + \delta \sqrt{C^i_t} dW^i_t,\quad C^i_0\label{cir}
= c_i,
\end{align}
where $(W^i)_{i\geq 1}$ are independent standard Brownian motions,
$\theta$ is the long-term average cost,
$k$ is the rate of mean reversion and $\delta$ determines the
variability of the cost process. 
The conventional producers aim
  to exit the market at the optimal time $\tau_i$. We
denote by $\omega^n_t(dx)$ the distribution of costs of conventional
producers who have not yet exited the market
(when there are a total of $n$ producers in the beginning of the game). In other words,
$$
\omega^n_t(dx) = \frac{1}{n}\sum_{i=1}^n
\delta_{C^i_t}(dx) \mathbf 1_{\tau_i >t}. 
$$
Note that we have normalized the distribution by the number of
producers, which turns out to be convenient at a later stage. {This is
equivalent} to assuming that the capacity of each producer equals
$\frac{1}{n}$. 
\paragraph{Renewable producers} Renewable producers aim to enter the market
  at the optimal time $\sigma_i$. To enter the market (build the power plant) they pay the cost $K_R$ after
  which the plant generates $S^i_t\in (0,1)$ units of electricity per unit
  time at zero cost, where the intermittent output $S^i$ follows the Jacobi process
\begin{align}
dS^i_t = \bar k(\bar \theta- S^i_t) dt + \bar\delta \sqrt{S^i_t(1- S^i_t)}
d\overline W^i_t,\quad S^i_0
= s_i\in (0,1),\label{jacobi}
\end{align}
{where $(\overline W_i)_{ i\geq 1}$ are standard Brownian 
motions}, independent from each other and from the price processes of
conventional producers. With this model we are not attempting to describe the
high-frequency variation of the power output due to the intermittency
of the renewable resource, but rather the slow variation of the
capacity factor due to climate variability and other effects.

We assume that there are $n$ potential renewable projects in
the beginning of the game, which may enter the market at some time. We
denote by $\eta^n_t(dx)$ the (potential) distribution of
output of renewable producers \emph{who have not yet entered the
  market}:
$$
\eta_t(dx) = \frac{1}{n}\sum_{i=1}^n
\delta_{S^i_t}(dx) \mathbf 1_{\sigma_i >t}. 
$$ 
We shall also need the distribution of output values of all renewable
projects:
$$
\bar \eta^n_t(dx) = \frac{1}{n}\sum_{i=1}^n
\delta_{S^i_t}(dx). 
$$ 
Note that both distributions are also normalized by the number of
agents. The normalization allows to simplify some expressions but our
model can easily be extended to the case when the number of
conventional producers is different from that of the renewable
projects, or when the capacities of these two types of agents are
different. 

%The Brownian motions driving the production costs of the conventional
%producers and those driving the output processes of renewable producers, are
%assumed to be independent.
We denote by $\mathcal L$ the infinitesimal
generator of the process $C^i$ and by $\overline{\mathcal L}$ that of $S^i$. These are
given by
\begin{align*}
\mathcal L f &= k(\theta-x) \frac{\partial f}{\partial x}
  +\frac{\delta^2 x}{2} \frac{\partial^2 f}{\partial x^2},\\
\overline{\mathcal L} f &= \bar k(\bar\theta-x) \frac{\partial f}{\partial x}
  +\frac{\bar\delta^2 x(1-x)}{2} \frac{\partial^2 f}{\partial x^2}.
\end{align*}
The state
spaces are denoted by $\Omega$ and $\overline \Omega$, respectively.

%In our model, the electricity price is determined by matching
%exogeneous demand to the supply function of the producers, in a
%stylized version of the day-ahead electricity market. 
\paragraph{Gain function of conventional producers}
% {\color{red}For example, one can take
%$$
%F(y) = \mathbf 1_{y\geq \varepsilon} +
%\frac{y}{\varepsilon}\mathbf 1_{0\leq y<\varepsilon}. 
%$$
%In other words, with this choice of $F$, the conventional producers bid their full capacity if
%the production cost is less than $p-\varepsilon$ and gradually reduce their bid
%as the cost approaches the price level. }

{ At time $t$ and price level $p$, each conventional producer solves
  the profit maximisation problem \[ \max_{\xi} p \xi - \int_0^\xi
    C^i_t(\xi) d\xi.\] The optimal fraction is  $ \xi^\ast := c^{-1}(p
  - C^i_t)$ if $p\geq C^i_t$. Thus, the profit of the producer is \[
    p \xi^\ast - \int_0^{\xi^\ast} C^i_t(\xi) d\xi = (p-C^i_t) F(p-C^i_t) -
                             \int_0^{F(p-C^i_t)} c(\xi) d\xi:= G(p - C^i_t),\]
where 
  \begin{align*}G(x) &= x F(x) -
                             \int_0^{F(x)} c(\xi) d\xi = \int_0^{F(x)} (x-c(\xi) ) d\xi= \int_0^x F(z) dz,
 \end{align*}                          
where the second equality results from a change of variable using the
fact that $F$ is the inverse of $c$. }

The problem of the
  individual conventional producer is then
$$
{\max_{\tau\in \mathcal T([0,T])} \mathbb E\left[\int_0^{\tau} e^{-\rho t}(G(P^n_t - C^i_t)-\kappa_C) dt + K_C e^{-(\gamma_C+\rho) \tau} \right],}
$$
where {$\mathcal{T}([0,T])$ represents the set of stopping times with respect to the filtration generated by $(W_t^i)_{i \geq 1}$,$(\bar{W}^i)_{i \geq 1}$, with values between $0$ and $T$}.
Here $P^n_t$ is the market price of electricity, determined from the
strategies of all agents via the merit order mechanism described below
in this section,  $\kappa_C$ is the fixed cost per
unit of time that the producer pays until exiting the market and
$e^{-\gamma_C t}  K_C$ is the  value recovered if the plant is
sold at time $t$, where $\gamma_C$ is the depreciation rate of the
conventional plant. To avoid a discontinuity at the terminal date $T$, we assume that
the plant is sold at date $T$ in any case. 
This maximization problem may equivalently be
written as
$$
{\max_{\tau\in \mathcal T([0,T])}} \mathbb E\left[\int_0^{T} e^{-\rho t} \mathbf
  1_{t< \tau}\left\{G(P^n_t - C^i_t)- \kappa_C-e^{-\gamma_C t} K_C(\rho +
    \gamma_C)\right\} dt \right],
$$
or, with a more compact notation, 
\begin{align*}
&{\max_{\tau\in \mathcal T([0,T])}} \mathbb E\left[\int_0^{T} e^{-\rho t} \mathbf
  1_{t< \tau}\left\{G(P^n_t - C^i_t)+f_C(t)\right\} dt \right],\\
&f_C(t) = - \kappa_C-e^{-\gamma_C t} K_C(\rho +
    \gamma_C). 
\end{align*}

\paragraph{Gain function of renewable producers}
The renewable producers always bid their
  full (but intermittent) capacity. We assume that the prices are
  always positive, and neglect the forecasting error of the
  producers. 
The problem of the individual agent is then
$$
{\max_{\sigma\in \mathcal{T}([0,T])} \mathbb E\left[\int_{\sigma}^T e^{-\rho t}
  (P^n_tS^i_t - \kappa_R) dt -K_R
e^{-\rho \sigma }+K_R e^{-\rho T - \gamma_R(T -
  \sigma)}\right],}
$$
where $\sigma$ is the time of entering the market, $\kappa_R$ is the
fixed cost of owning the plant, $K_R$ is the fixed cost of building
the plant and {$K_R ^{ - \gamma_R(T -
  \sigma)}$} is the value of the plant at time $T$, where
$\gamma_R$ is the depreciation rate of the renewable plant. Once
again, to minimize the boundary effect, we assume that
the agent recovers the value of the plant reduced by the depreciation
rate at the terminal date $T$. This
maximization problem may equivalently be written as 
{$$
\max_{\sigma\in \mathcal{T}([0,T])} \mathbb E\left[\int_0^T \mathbf
  1_{t<\sigma}e^{-\rho t}\left\{
-P^n_t S^i_t + \kappa_R + \rho K_R + \gamma_R K_R
e^{-(\rho+\gamma_R)(T-t)} 
\right\}\right],
$$}
or, with a more compact notation,
{\begin{align*}
&\max_{\sigma\in \mathcal{T}([0,T])} \mathbb E\left[\int_0^T \mathbf
  1_{t<\sigma}e^{-\rho t}\left\{
-P^n_t S^i_t +f_R(t)\right\}\right],\\
&f_R(t) = \kappa_R + \rho K_R+ \gamma_R K_R
e^{-(\rho+\gamma_R)(T-t)}. 
\end{align*}}

\paragraph{Baseline supply} We assume that in addition to the
renewable and conventional producers considered above there is a
baseline supply  by conventional producers, which will never leave
the market (e.g., state-owned producers, which ensure the network
security) and do not take part in the game. This baseline supply at price level $p$ is denoted by
$F_0(p)$ where the function $F_0$ is increasing and satisfies $F_0(0)=0$, $F_0(p+h)-F(p)\geq c
h$ for some $c>0$ and all $p,h\geq 0$. This last assumption is imposed
for techical reasons and ensures the continuity of the equilibrium
price process. 

The total supply by the conventional producers at price level $p$,
including the baseline supply, is
therefore given
by 
$$
\int_{\Omega} F(p-x) \omega^n_t(dx) + F_0(p).
$$ 
The total renewable supply at time $t$,
on the other hand, is given by
$$
{R^n_t = \int_{\overline \Omega}
x(\bar \eta^n_t(dx)-\eta^n_t(dx))}.
$$ 
\paragraph{Price formation}

In our model, the different agents are coupled
through the market price, determined by matching the exogeneous demand process
$\overline D_t$, to the aggregate supply function of market
participants, in a stylized version of the day-ahead electricity market. In other words, since the residual demand after
subtracting the renewable supply which is always present in the
market, must be
met by conventional producers, the electricity price $P^n_t$ is defined
as follows. 
\begin{align*}
P^n_t = \inf\{P: (\overline D_t - R^n_t)^+ \leq \int_{\Omega} F(P-x)
  \omega^n_t(dx) + F_0(P)\}\wedge \overline P,
\end{align*}
where $\overline P$ is the price cap in the market and we use the
convention $\inf\emptyset = +\infty$. When
the price cap $\overline P$ is reached, the demand may not be entirely
satisfied by the producers. 

Our aim here is not to model the daily price fluctuations but rather the
slow evolution of the average market price due to changes of structure
of 
electricity supply. However, to make the illustration more realistic,
in section \ref{illustration.sec} we shall consider separately the peak
price (Mon-Fri, 7AM-8PM) and the off-peak price. This extension
presents no technical difficulties, and therefore, to simplify
notation, we consider a single price in the rest of the paper. 

\section{Relaxed MFG formulation of the problem}
\label{relaxed.sec}
To simplify the resolution and the study of price equilibria, we place
ourselves from now on in the MFG framework, where the
number of agents (both conventional and renewable) is assumed to be
infinite. As in many papers on mean-field games, we do not study the
convergence of the $n$-player game to the MFG but analyze
the MFG framework directly. We denote the limiting versions of the distributions
$\omega^n_t$, $\eta^n_t$ and $\bar\eta^n_t$ by $\omega_t$, $\eta_t$
and $\bar\eta_t$, respectively, and the limiting market price and
renewable demand by $P_t$ and $R_t$. 
Since in our game the idiosyncratic noises of agents are independent
and there is no common noise, these limiting measures and processes are assumed to be deterministic. 
Note that $\omega_t$ and $\eta_t$ are
not probability distributions: the total mass of both these measures is
decreasing with $t$. 
%To avoid
%technicalities, we assume that $\eta_t$ and $\omega_t$ are supported
%on $[0,\infty)$. 
We finally assume that the demand $\overline D_t$
is deterministic. 

We follow the relaxed optimal stopping MFG approach, introduced in
\cite{bouveret2018mean}, adapting it to the present setting of electricity
markets. To this end, we first recall the topology on flows of
measures used in this reference. {Let $V_\Omega$ be the space of flows of signed bounded
measures on ${\Omega}$, $(m_t(\cdot))_{0\leq t \leq
  T}$ such that: for every $t\in [0,T]$, $m_t$ is a signed bounded
measure on ${\Omega}$, for every $A\in \mathcal B({\Omega})$, the
mapping $t\mapsto m_t(A)$
is measurable, and $\int_0^T \int_{{\Omega}} m_t(dx)\, dt <\infty$. To each
flow $m\in V_\Omega$, we associate a signed measure on $[0,T]\times {\Omega}$ defined by $\mu(dt,dx):= m_t(dx)\,
dt$, and we endow $V_\Omega$ with the topology of weak
convergence of the associated measures. $V_{\overline\Omega}$ is
defined in the same way.}

Given a deterministic measurable price process $(P_t)_{t\geq 0}$, 
the relaxed solution approach consists in replacing the optimal stopping
problem of
individual conventional producer, 
\begin{align}
\sup_\tau\mathbb E\left[\int_0^T e^{-\rho t}[G(P_t - C_t) +f_C(t)] \mathbf 1_{t< \tau }dt \right],\label{optstrict}
\end{align}
by its relaxed version
\begin{align}
\sup_{\omega \in \mathcal A(\omega_0)} \int_0^T \int_\Omega e^{-\rho t} [G(P_t-x)
  +f_C(t)] \omega_t(dx) \,dt,\label{ind.conv}
\end{align}
where the set
$\mathcal A(\omega_0)$ contains all flows of positive bounded measures
$(\hat \omega_t)_{0\leq t\leq T} \in V_\Omega$ satisfying 
$$
\int_\Omega u(0,x)\omega_0(dx)  + \int_0^T \int_{\Omega} \left\{\frac{\partial u}{\partial t} +
  \mathcal L u\right\} \hat \omega_t(dx) \, dt \geq 0
$$
for all $u\geq 0, u \in C^{1,2}([0,T]\times \Omega)$ such that $\frac{\partial
  u}{\partial t} + \mathcal Lu$ is bounded.

In the relaxed formulation of the optimal stopping problem, instead of
looking for an optimal stopping time, one looks for the optimal
\emph{measure flow}, corresponding, at each time $t$, to the
distribution of agents in a population which have not yet exited the
game. The precise relationship between the optimal stopping problem \eqref{optstrict}
and its relaxed version \eqref{ind.conv} is described in
\cite{bouveret2018mean}. In particular, under appropriate assumptions,
it holds that 
$$
\sup_{\omega \in \mathcal A(\omega_0)} \int_0^T \int_\Omega e^{-\rho t} [G(P_t-x)
  +f_C(t)] \omega_t(dx) \,dt = \int_{\Omega} v(0,x) \omega_0(dx),
  $$
  where $v(0,x)$ is the value function of the optimal stopping problem
  \eqref{optstrict} at time $t=0$:
  $$
v(0,x)  = \sup_\tau\mathbb E\left[\int_0^T e^{-\rho t}[G(P_t - C^{(0,x)}_t) +f_C(t)] \mathbf 1_{t< \tau }dt \right]
 $$

Similarly, the optimal stopping problem of individual renewable producer 
\begin{align*}
\sup_\sigma\mathbb E\left[\int_0^T e^{-\rho t}\left\{-P_t S_t +f_R(t)\right\}\mathbf
  1_{t< \sigma }dt \right] 
\end{align*}
is replaced by its relaxed version
\begin{align}
\sup_{\eta\in \overline{\mathcal A}(\eta_0) }\int_0^T
  \int_{\overline\Omega} e^{-\rho t}\left\{-P_t x + f_R(t)\right\}\eta_t(dx)\, dt, \label{ind.ren}
\end{align}
where the set $\overline{\mathcal A}(\eta_0)\subset V_{\overline \Omega}$ is defined similarly to
${\mathcal A}(\omega_0).$

Given the flows $(\omega_t)_{0\leq t\leq T}$ and $(\eta_t)_{0\leq t\leq T}$,  the price
process $(P_t)_{0\leq t\leq T}$ is defined as follows. 
\begin{align*}
P_t = \inf\{P: (\overline D_t - R_t)^+ \leq \int_\Omega F(P_t-x)
  \omega_t(dx) + F_0(P_t)\}\wedge \overline P,\quad 0\leq t\leq T,
\end{align*}
where ${R_t = \int_{\overline\Omega}
x(\bar \eta_t(dx)-\eta_t(dx))}$. We denote the price, defined in
this way, by $P_t(\omega_t, \eta_t)$.\\

We now introduce the definition of a relaxed Nash equilibrium.
\begin{definition}\label{defnash}
The Nash equilibrium of the relaxed mean-field game is the couple $(\omega^*_t,\eta^*_t)$
such that for any other measure $\omega \in \mathcal A(\omega_0)$, 
\begin{multline*}
\int_0^T \int_{\Omega} e^{-\rho t} [G(P_t(\omega^*_t,\eta^*_t)-x)
  +f_C(t)] \omega_t(dx) \,dt \\\leq \int_0^T \int_{\Omega} e^{-\rho t} [G(P_t(\omega^*_t,\eta^*_t)-x)
  +f_C(t)] \omega^*_t(dx) \,dt,
\end{multline*}
and for any other measure $\eta \in \overline{\mathcal A}(\eta_0)$
\begin{multline*}
\int_0^T \int_{\overline\Omega} e^{-\rho t}  [-P_t(\omega^*_t,\eta^*_t) x +
f_R(t)]\eta_t(dx)\, dt \\\leq \int_0^T \int_{\overline\Omega} e^{-\rho t}  [-P_t(\omega^*_t,\eta^*_t) x +
f_R(t)]\eta^*_t(dx)\, dt. 
\end{multline*}
\end{definition}

In the rest of this section we study the solutions of the relaxed MFG
problem. 
Firsly, the following lemma establishes the existence of solution for the
individual relaxed optimal stopping problems. 
\begin{lemma}\label{pfixed}
Let the price process $(P_t)_{0\leq t\leq T}$ be fixed, and asume that
it has bounded variation on $[0,T]$.
\begin{itemize}
\item[i.] Let $\omega_0$ satisfy 
$$
\int_\Omega \ln\{1+|x|\}\omega_0(dx) <\infty. 
$$
Then there exists $\omega^* \in
\mathcal A(\omega_0)$ which solves the problem \eqref{ind.conv}.
\item[ii.] Let $\eta_0$ satisfy 
$$
\int_{\overline\Omega} \ln\{1+|x|\}\eta_0(dx) <\infty. 
$$ 
Then there exists $\eta^* \in
\overline{\mathcal A}(\eta_0)$ which solves the problem \eqref{ind.ren}.
\end{itemize}
\end{lemma}
\begin{proof}
Part i. 
Choose a maximizing sequence of flows of measures
$(\omega^n_t)_{t\in[0,T]}^{n\geq 1}\subset \mathcal A(\omega_0)$.
By Lemma 3.8 in \cite{bouveret2018mean}, the set $\mathcal A(\omega_0)$ is
sequentially compact\footnote{Our processes do
  not satisfy Assumption (X-SDE) of \cite{bouveret2018mean}, but Lemma 3.8 only needs linear growth of coefficients}, therefore this sequence has a subsequence, also
denoted by $(\omega^n)$, which converges to a limit $\omega^*
\in \mathcal A(m^*_0)$. It remains to show that $\omega^*$ is a
maximizer of \eqref{ind.conv}. Let $(P^m_t)^{m\geq 1}$ be a sequence
of uniformly bounded continuous mappings approximating $P_t$ in $L^1([0,T])$. Note that $G$ is a Lipschitz function, with Lipschitz
constant $1$. Then,
\begin{align}
\int_0^T dt\int_\Omega [G(P_t-x)+f_C(t)] \omega^n_t(dx) &= \int_0^T dt\int_\Omega
[G(P^m_t-x) +f_C(t)] \omega^n_t(dx) \notag\\ &+ \int_0^T dt\int_\Omega
[G(P_t-x)-G(P^m_t-x)] \omega^n_t(dx).\label{2terms}
\end{align}
The second term satisfies
\begin{align}
\left|\int_0^T dt\int_\Omega
[G(P_t-x)-G(P^m_t-x)] \omega^n_t(dx)\right| \leq \int_0^T |P_t -
P^m_t| \int_\Omega \omega^n_t(dx)\label{2ndterm}
\end{align}
Taking the test function 
$$
u(t,x) = \int_t^T f(s) ds
$$
with $f\geq 0$ continuous in the definition of $\mathcal A(\omega_0)$,
we have:
$$
\int_\Omega \omega_0(dx) \int_0^T f(s) ds \geq \int_0^T f(s) ds
\int_\Omega \omega^n_t(dx).
$$
This implies that for every $n$, $t$-almost everywhere, 
$$
\int_\Omega \omega^n_t(dx) \leq \int_\Omega \omega_0(dx),
$$
and so \eqref{2ndterm} converges to zero as $m\to \infty$, uniformly
on $n$. On the other hand, by weak convergence, the first term in
\eqref{2terms} converges to 
$$
\int_0^T dt\int_\Omega
[G(P^m_t-x) +f_C(t)] \omega^*_t(dx),
$$
which, once again, converges to 
$$
\int_0^T dt\int_\Omega
[G(P_t-x) +f_C(t)] \omega^*_t(dx)
$$
as $m\to \infty$. Combining the three limits and using the fact that
the convergence of \eqref{2ndterm} is uniform on $n$, the proof is
completed.\\

Part ii. This part is shown similarly to the first part. 
% Indeed, for two instants $t_1$ and $t_2$ with
% $0\leq t_1 < t_2\leq T$, we have
% \begin{align*}
% |f^n(t_2) - f^n(t_1)| &\leq \left|\int_0^\infty (P_{t_2} -x)^+
% \omega^n_{t_1}(dx) - \int_0^\infty (P_{t_1} -x)^+ \omega^n_{t_1}(dx)\right|
% \\&\qquad + \left|\int_0^\infty (P_{t_2} -x)^+
% \omega^n_{t_2}(dx) - \int_0^\infty (P_{t_2} -x)^+
%     \omega^n_{t_1}(dx)\right|
% \end{align*}
% The first term in the right-hand side is bounded from above by 
% $$
% |P_{t_2} - P_{t_1}| \omega_{t_1}^n ((0,\infty))\leq |P_{t_2} - P_{t_1}|\omega_0((0,\infty))
% $$
% The second term, on the other hand, is evaluated as 
\end{proof}

In the mean-field game context, the price becomes a function of
$\eta_t$ and $\omega_t$.
%  $P_t = p(t,\omega_t,\eta_t)$ defined by
% $$
% p(t,\omega,\eta) = F_{\omega+\bar \omega}^{-1}\left(\left\{\overline D_t - \int_{\mathbb R_+}
% x(\bar \eta_t(dx)-\eta(dx))\right\}_+\right)\wedge \overline P,
% $$ 
% where $F_\omega$ denotes the distribution function of the measure
% $\omega$. 
The following technical lemma establishes the properties of the price
process. Its proof is presented in the appendix. 
\begin{lemma}\label{bvariation}${}$
\begin{itemize}
\item[i.] Let $\omega \in \mathcal A(\omega_0)$ and $\eta \in \overline{\mathcal
A} (\eta_0)$ and assume that the demand $\overline D$ has bounded
variation on $[0,T]$. Then the price process $P_t(\omega_t,\eta_t)$ has
bounded variation on $[0,T]$ as well. 
\item[ii.] Let $(\omega^n_t)$ and $(\eta^n_t)$ be sequences of
  elements of $\mathcal A (\omega_0)$ and $\overline{\mathcal
    A}(\eta_0)$, converging, respectively, to $\omega^* \in \mathcal A
  (\omega_0)$ and $\eta^*\in \overline{\mathcal
    A}(\eta_0)$.  Assume that the demand $\overline D$ has bounded
variation on $[0,T]$. Then there exists a subsequence $(n_k)_{k\geq
  0}$ such that  $P_t(\omega^{n_k}_t,\eta^{n_k}_t)$ converges to
  $P_t(\omega^*_t,\eta^*_t)$ in $L^1([0,T])$. 
\end{itemize}
\end{lemma}
% \paragraph{Potential games} To prove that the game is a potential game, we need to find a
% functional $\mathcal P(t,\omega,\eta)$ such that 
% $$
% \nabla_\omega \mathcal P = (p(t,\omega,\eta) - x)^+ - \rho K_C,\quad
% \text{and}\quad \nabla_\eta \mathcal P = - p(t,\omega,\eta) x + \rho G.
% $$
% It is then clear that this game will not be potential since the
% equality 
% $$
% \nabla^2_{\omega \eta} \mathcal P = \nabla^2_{\eta\omega} \mathcal P
% $$
% may not hold. 

We now prove the existence of a relaxed Nash equilibrium.
\begin{proposition}
There exists a Nash equilibrium for the relaxed MFG problem.
\end{proposition}
\begin{proof}
Following the ideas of Theorem 4.4 in \cite{bouveret2018mean}, we define the set valued mapping
\begin{align*}
\Theta: (\omega,\eta)\mapsto &\argmax_{\hat\omega \in \mathcal A(\omega_0)}\int_0^T \int_{\Omega} e^{-\rho t} [G(P_t(\omega_t,\eta_t)-x)
  +f_C(t)] \hat\omega_t(dx) \,dt\\
&\times \argmax_{\hat\eta \in \overline{\mathcal A}(\eta_0)}\int_0^T \int_{\overline\Omega} e^{-\rho t}  [-P_t(\omega_t,\eta_t) x +
f_R(t)]\hat\eta_t(dx)\, dt. 
\end{align*}
To establish existence of Nash equilibrium by applying Fan-Glicksberg
fixed point theorem, it is enough to show that $\Theta$ has closed
graph, which is defined by
$$
\text{Gr}(\Theta) = \{(\omega,\bar\omega,\eta,\bar\eta)\in \mathcal
A(\omega_0)^2 \times \overline{\mathcal A}(\eta_0)^2: (\omega,\eta) \in
\Theta(\bar\omega,\bar\eta)\}
$$
To prove that $\text{Gr}(\Theta)$ is closed it is in turn sufficient to show that
for any two sequences $(\omega^n,\bar\omega^n)_{n\geq 1} \in
\mathcal{A}(\omega_0)^2$, $(\eta^n,\bar \eta^n)_{n\geq 1} \in \bar{\mathcal{A}}(\eta_0)^2$, which converge weakly to $(\omega, \bar{\omega}) \in (\mathcal{A}(\omega_0))^2$ (resp.$(\eta, \bar{\eta}) \in (\bar{\mathcal{A}}(\eta_0))^2$ ) and such that
\begin{multline*}
\int_0^T dt \int_{\Omega} e^{-\rho t} [G(P_t(\omega^n_t,\eta^n_t)-x)
  +f_C(t)] \bar{\omega}^n_t(dx) \nonumber \\ \geq \int_0^T dt \int_{\Omega} e^{-\rho t} [G(P_t(\omega^n_t,\eta^n_t)-x)
  +f_C(t)] \hat{\omega}_t(dx), \,\, {\rm for \,\, all\,\,} \hat{\omega} \in \mathcal{A}(\omega_0)
\end{multline*}
and
\begin{multline*}
 \int_0^T \int_{\overline\Omega} e^{-\rho t}  [-P_t(\omega^n_t,\eta^n_t) x +
f_R(t)]\bar{\eta}^n_t(dx)\, dt \nonumber \\ \geq \int_0^T dt \int_{\overline\Omega} e^{-\rho t}  [-P_t(\omega^n_t,\eta^n_t) x +
f_R(t)]\hat{\eta}_t(dx), \,\, {\rm for \,\, all\,\,} \hat{\eta} \in \bar{\mathcal{A}}(\eta_0),
\end{multline*}
we have 
\begin{multline*}
\int_0^T dt \int_{\Omega} e^{-\rho t} [G(P_t(\omega_t,\eta_t)-x)
  +f_C(t)] \bar{\omega}_t(dx) \nonumber \\ \geq \int_0^T dt \int_{\Omega} e^{-\rho t} [(G(P_t(\omega_t,\eta_t)-x)
  +f_C(t)] \hat{\omega}_t(dx), \,\, {\rm for \,\, all\,\,} \hat{\omega} \in \mathcal{A}(\omega_0)
\end{multline*}
and
\begin{multline*}
 \int_0^T \int_{\overline\Omega} e^{-\rho t}  [-P_t(\omega_t,\eta_t) x +
f_R(t)]\bar{\eta}_t(dx)\, dt \nonumber \\ \geq \int_0^T dt \int_{\overline\Omega} e^{-\rho t}  [-P_t(\omega_t,\eta_t) x +
f_R(t)]\hat{\eta}_t(dx), \,\, {\rm for \,\, all\,\,} \hat{\eta} \in \bar{\mathcal{A}}(\eta_0).
\end{multline*}
To prove this, it is enough to show, that up to taking a subsequence,
\begin{align}
&\lim_n \int_0^T dt \int_{\Omega} e^{-\rho t} [G(P_t(\omega^n_t,\eta^n_t)-x)
  +f_C(t)] \bar{\omega}^n_t(dx) \notag\\ &\qquad  = \int_0^T dt \int_{\Omega} e^{-\rho t} [G(P_t(\omega_t,\eta_t)-x)
  +f_C(t)] \bar{\omega}_t(dx) \label{lim1}\\
&\lim_n  \int_0^T \int_{\overline\Omega} e^{-\rho t}  [-P_t(\omega^n_t,\eta^n_t) x +
f_R(t)]\bar{\eta}^n_t(dx)\, dt \notag\\ &\qquad =  \int_0^T \int_{\overline\Omega} e^{-\rho t}  [-P_t(\omega_t,\eta_t) x +
f_R(t)]\bar{\eta}_t(dx)\, dt. \label{lim2}
\end{align} 
We first focus on \eqref{lim2}, which can be rewritten as follows.
\begin{align*}
&\lim_n  \int_0^T e^{-\rho t} dt \left\{
  -P_t(\omega^n_t,\eta^n_t)\int_{\overline\Omega}  x \bar{\eta}^n_t(dx) +f_R(t)
  \int_{\overline\Omega} \bar{\eta}^n_t(dx)\right\}
\\ &\qquad =  \int_0^T e^{-\rho t} dt \left\{
  -P_t(\omega_t,\eta_t)\int_{\overline\Omega}  x \bar{\eta}_t(dx) +f_R(t)
  \int_{\overline\Omega} \bar{\eta}_t(dx)\right\}.
\end{align*}
Using the same arguments as in the proof of Lemma
\ref{bvariation} (step 1), one can prove that the total variation of
the map $t\mapsto \int_{\overline\Omega}
x \bar{\eta}_t^n(dx)$ on $[0,T]$ is uniformly bounded with respect to
$n$. This implies that one can find a subsequence of this sequence of maps, converging in
$L^1([0,T])$ to some limit, which can be identified, due to weak
convergence of the sequence $\bar{\eta}^n$, with $\int_{\overline\Omega} x
\bar{\eta}_t(dx)$. Furthermore, by Lemma \ref{bvariation}, part ii.,
$P_\cdot(\omega^n_\cdot,\eta^n_\cdot)$ converges in $L^1([0,T])$ to
$P_\cdot(\omega_\cdot, \eta_\cdot)$ (up to taking a
subsequence). Since both factors are bounded, the integral of their
product also converges, and the convergence of the last term in
\eqref{lim2} follows from weak convergence of the measures. 

Let us now turn to \eqref{lim1}. Recall that $G$ is Lipschitz with
constant $1$. Then, 
\begin{align*}
&\Big| \int_0^T dt \int_{\Omega} e^{-\rho t} [G(P_t(\omega^n_t,\eta^n_t)-x)
  +f_C(t)] \bar{\omega}^n_t(dx) \notag\\ &\qquad  - \int_0^T dt \int_{\Omega} e^{-\rho t} [G(P_t(\omega_t,\eta_t)-x)
  +f_C(t)] \bar{\omega}_t(dx) \Big|\\
&\leq   \int_0^T dt e^{-\rho t}
  |(P_t(\omega^n_t,\eta^n_t) - P_t(\omega_t,\eta_t)| \int_{\Omega}\bar{\omega}^n_t(dx)\notag\\ & + \Big| \int_0^T dt \int_{\Omega} e^{-\rho t} [G(P_t(\omega_t,\eta_t)-x)
  +f_C(t)] (\bar{\omega}^n_t - \bar{\omega}_t(dx)) \Big|.
\end{align*}
As before, we can show that, up to taking a subsequence 
\begin{align}
\int_{\Omega}\bar{\omega}^n_\cdot(dx)\xrightarrow{L^1([0,T])} \int_{\Omega}\bar{\omega}_\cdot(dx),\label{convintom}
\end{align}
and thus the first term above converges to $0$ since $P_t$ is bounded. For the second term,
we consider a sequence of bounded continuous functions $P^m_t$, approximating
the price $P_t(\omega_t,\eta_t)$ in $L^1([0,T])$. Using once again the
Lipschitz property of $G$ and the fact that this function is increasing, 
\begin{align*}
&\Big| \int_0^T dt \int_{\Omega} e^{-\rho t} [G(P_t(\omega_t,\eta_t)-x)
  +f_C(t)] (\bar{\omega}^n_t - \bar{\omega}_t(dx)) \Big|\\
&\leq \Big| \int_0^T dt \int_{\Omega} e^{-\rho t} [G(P_t^m-x)
  +f_C(t)] (\bar{\omega}^n_t - \bar{\omega}_t(dx)) \Big|\\
  & + \int_0^T dt \, e^{-\rho t} |(P_t(\omega_t,\eta_t) -P^m_t| \left|\int_{\Omega} (\bar{\omega}^n_t - \bar{\omega}_t(dx))\right|.
\end{align*} 
The first term above converges to zero in view of the weak convergence
of measures, and for the second term we can once again use
\eqref{convintom} and dominated convergence. 
\end{proof}
\paragraph{Uniqueness of the equilibrium price process}
We will show that different Nash equilibria necessarily correspond to
the same price, except possibly on a set of measure
zero. 
\begin{proposition}
Let $(\omega^1,\eta^1)$ and
$(\omega^2,\eta^2)$ be two Nash equilibria. Then, the set of points
$t$ such that $P_t(\omega^1_t,\eta^1_t) \neq P_t(\omega^2_t,\eta^2_t)$ has
Lebesgue measure zero. 
\end{proposition}
\begin{proof}
Let 
$$
f(t,\omega,\eta,x) = G(P_t(\omega,\eta)-x) +f_C(t)\quad
\text{and}\quad \bar f(t,\omega,\eta,x) = -P_t(\omega,\eta) x +f_R(t). 
$$
By definition of the Nash equilibrium, 
\begin{multline*}
\int_0^T \int_{\Omega} e^{-\rho t}
  (f(t,\omega^1_t,\eta^1_t,x)-f(t,\omega^2_t,\eta^2_t,x)) (\omega^1_t(dx) -
  \omega^2_t(dx)) dt \\
+  \int_0^T \int_{\overline\Omega} e^{-\rho t} (\bar
f(t,\omega^1_t,\eta^1_t,x)-\bar f(t,\omega^2_t,\eta^2_t,x)) (\eta^1_t(dx) -
  \eta^2_t(dx))dt \geq 0.
\end{multline*}
Choose $t$ such that $P_t(\omega^1_t,\eta^1_t)> P_t(\omega^2_t,\eta^2_t)$
From the fact that $F$ is increasing and the mean value theorem, we
deduce the following simple estimate:
\begin{align*}
(P_1-P_2) F(P_2-x)\leq G(P_1-x) - G(P_2-x)\leq (P_1-P_2)F(P_1-x),\quad
  P_1\geq P_2.
\end{align*}
Moreover, the definition
of the price implies that 
\begin{align*}
&\int_\Omega
  F(P_t(\omega^1_t,\eta^1_t) - x) \omega^1_t(dx)\leq  (\overline D_t -
R^1_t)^+  - F_0(P_t(\omega^1_t,\eta^1_t))\\
&\int_\Omega
  F(P_t(\omega^2_t,\eta^2_t) - x) \omega^2_t(dx)  = (\overline D_t -
R^2_t)^+  - F_0(P_t(\omega^2_t,\eta^2_t))\\
&\overline D_t - R^1_t>0,
\end{align*} 
where
$$
R^i_t = \int_{\overline \Omega} x(\bar \eta_t (x) - \eta^i_t(x)) dx.
$$ 
Therefore,
\begin{align*}
&\int_{\Omega} 
  (G(P_t(\omega^1_t,\eta^1_t)-x)-G(P_t(\omega^2_t,\eta^2_t)-x)) (\omega^1_t(dx) -
  \omega^2_t(dx))\\ &\leq  (P_t(\omega^1_t, \eta^1_t) -
  P_t(\omega^2_t,\eta^2_t)) \\ &\qquad \qquad \times \left\{\int_\Omega
  F(P_t(\omega^1_t,\eta^1_t) - x) \omega^1_t(dx) - \int_\Omega
  F(P_t(\omega^2_t,\eta^2_t) - x) \omega^2_t(dx)\right\}\\
&\leq  (P_t(\omega^1_t, \eta^1_t) -
  P_t(\omega^2_t,\eta^2_t))\\&\qquad \qquad \times\left\{(\overline D_t -
R^1_t)^+  - F_0(P_t(\omega^1_t,\eta^1_t))- (\overline D_t -
R^2_t)^+  + F_0(P_t(\omega^2_t,\eta^2_t))\right\}
\end{align*}
and we obtain the antimonotonicity property
\begin{align*}
&\int_{\Omega} 
  (f(t,\omega^1_t,\eta^1_t,x)-f(t,\omega^2_t,\eta^2_t,x)) (\omega^1_t(dx) -
  \omega^2_t(dx)) \\
&\qquad +   \int_{\overline\Omega} (\bar
f(t,\omega^1_t,\eta^1_t,x)-\bar f(t,\omega^2_t,\eta^2_t,x)) (\eta^1_t(dx) -
  \eta^2_t(dx)) \\
& \leq \int_{\Omega} 
  ((P_t(\omega^1_t,\eta^1_t)-x)F(P_t(\omega^1_t,\eta^1_t)-x)-(P_t(\omega^2_t,\eta^2_t)-x)
                     F(P_t(\omega^2_t,\eta^2_t)-x)) \\ &\qquad \qquad \times(\omega^1_t(dx) -
  \omega^2_t(dx)) +  (P_t(\omega^2_t,\eta^2_t) -P_t(\omega^1_t,\eta^1_t))(R^1_t -
  R^2_t) \\
&\leq\left\{(\overline D_t -
R^1_t)^+ + R^1_t  - F_0(P_t(\omega^1_t,\eta^1_t))- (\overline D_t -
R^2_t)^+ -R^2_t + F_0(P_t(\omega^2_t,\eta^2_t))\right\}\\ &\qquad
                                                            \qquad
                                                            \times  (P_t(\omega^1_t, \eta^1_t) -
  P_t(\omega^2_t,\eta^2_t))\\
& = (P_t(\omega^1_t, \eta^1_t) -
  P_t(\omega^2_t,\eta^2_t))\left\{  - F_0(P_t(\omega^1_t,\eta^1_t))- (\overline D_t -
R^2_t)^- + F_0(P_t(\omega^2_t,\eta^2_t))\right\}\\
&\leq -c (P_t(\omega^1_t, \eta^1_t) -
  P_t(\omega^2_t,\eta^2_t))^2
\end{align*}
The case where $P_t(\omega^1_t,\eta^1_t)<
P_t(\omega^2_t,\eta^2_t)$ is dealt with by a symmetric
argument. Finally, when $P_t(\omega^1_t,\eta^1_t)=
P_t(\omega^2_t,\eta^2_t)$, the left-hand side of the above equality is
clearly zero. Thus,
\begin{multline*}
\int_0^T \int_{\Omega} e^{-\rho t}
  (f(t,\omega^1_t,\eta^1_t,x)-f(t,\omega^2_t,\eta^2_t,x)) (\omega^1_t(dx) -
  \omega^2_t(dx)) dt \\
+  \int_0^T \int_{\overline\Omega} e^{-\rho t} (\bar
f(t,\omega^1_t,\eta^1_t,x)-\bar f(t,\omega^2_t,\eta^2_t,x)) (\eta^1_t(dx) -
  \eta^2_t(dx))dt \\\leq -c \int_0^T e^{-\rho t} (P_t(\omega^1_t, \eta^1_t) -
  P_t(\omega^2_t,\eta^2_t))^2 dt,
\end{multline*}
which is in contradiction with the Nash equilibrium property unless $P_t(\omega^1_t, \eta^1_t) =
  P_t(\omega^2_t,\eta^2_t)$ almost everywhere on $[0,T]$. 
\end{proof}

% \paragraph{Relation with the strong formulation}
% {\color{red}RD rajoute une remarque sur l'absence de solutions mixtes}

% We introduce here the value function $v$ associated to the conventional producer, that is, for each $(t,c)$
% \begin{align}
% v(t,c)=\sup_{\tau \in \mathcal{T}_{t,T}} \mathbb{E}_{t,c}\left[\int_t^\tau e^{-\rho s} (P_s-C_s^{(t,c)})^+ds+K_Ce^{-\rho \tau}\right]
% \end{align}
% respectively the value function $\bar{v}$ associated to the renewable producer, which is defined, for each $(t,s)$ as follows
% \begin{align}
% \bar{v}(t,s)=\sup_{\sigma \in \mathcal{T}_{t,T}} \mathbb{E}_{t,s}\left[\int_t^\sigma e^{-\rho u} P_u S_u^{(t,s)})^+ds-K_Re^{-\rho \sigma}\right]
% \end{align}
% Let $(\omega^*, \eta^*)$ be a relaxed Nash equilibrium.
% By Lemma \todo{ADD REFERENCE}, we get 
% \begin{align}
% \int_\Omega v(0,c)\omega_0^*(dc)= \int_0^T \int_\Omega f(s,\omega^*,\eta^*,c)\omega_s^*(dc)ds
% \end{align}
% and
% \begin{align}
% \int_\Omega \bar{v}(0,s)\eta_0^*(ds)= \int_0^T \int_\Omega \bar{f}(u,\omega^*,\eta^*,s)\eta_s^*(ds)du.
% \end{align}
% Moreover, by the same Lemma, we get that 
% \begin{align}
% \int_{v=K}f(s,\omega^*,\eta^*,c)\omega_s^*(dc)ds=0
% \end{align}
% and 
% \begin{align}
% \int_{\bar{v}=G}\bar{f}(u,\omega^*,\eta^*,s)\eta_u^*(ds)du=0
% \end{align}
% \hrule
% \begin{align*}
% G_t(y) = \int_{\Omega} F(y-x) (\omega_t(dx)+\bar
%   \omega_t(dx))\\
% G'_t(y) = \frac{1}{\varepsilon}\int_{y-\varepsilon}^y (\omega_t(dx)+\bar
%   \omega_t(dx)) \leq \frac{1}{\varepsilon}\int_{y-\varepsilon}^y (\omega^*_0(dx)+\bar
%   \omega^*_0(dx))
% \end{align*}

\section{Numerical computation of the equilibrium measures}
\label{numerics.sec}
\paragraph{Computing the MFG equilibrium}
In the numerical algorithm, we successively take a step of descreasing
size towards the best response, until a desired convergence criterion
is met. Recall that the Nash equilibrium of the relaxed mean-field
game is described in Definition \ref{defnash}. The computation of
$\omega^*$ and $\eta^*$ is achieved using the following procedure. 
\begin{itemize}
\item Choose initial values $\omega^{(0)} \in \mathcal A(\omega_0)$
  and ${\eta^{(0)} \in \overline{\mathcal A}(\eta_0)}$
\item For $i=1\dots N_{iter}$
\begin{itemize}
\item Compute the best responses
\begin{align*}
\tilde \omega^{(i)} &= {\arg\max_{\omega\in \mathcal A(\omega_0)}} \int_0^T
\int_\Omega e^{-\rho t} \left[G(P_t(\omega^{(i-1)}, \eta^{(i-1)})-x) +
  f_C(t)\right]\omega_t(dx)\, dt,\\
\tilde \eta^{(i)} &= {\arg\max_{\eta\in\overline{ \mathcal A}(\eta_0)}} \int_0^T
\int_{\overline\Omega} e^{-\rho t} \left[-P_t(\omega^{(i-1)}, \eta^{(i-1)})x +
  f_R(t)\right]\eta_t(dx)\, dt.
\end{align*}
\item Choose the step size parameter: $w^{(i)} = \frac{1}{i}$. 
\item Update the measures:
\begin{align*}
\omega^{(i)}& = (1-w^{(i)}) \omega^{(i-1)} + w^{(i)}
  \tilde\omega^{(i)},\\
\eta^{(i)} &= (1-w^{(i)}) \eta^{(i-1)} + w^{(i)}
  \tilde\eta^{(i)}.    
\end{align*}
\end{itemize}
\end{itemize}
The number of steps of the algorithm may be fixed or chosen based on a
convergence criterion. In our implementation, we monitor the relative
improvement of the best response, defined by 
\begin{align*}
\mathcal E^C(\omega^{(n)},\eta^{(n)}) &= {\max_{\omega\in \mathcal A(\omega_0)}} \int_0^T
\int_\Omega e^{-\rho t} \left[G(P_t(\omega^{(n)}, \eta^{(n)})-x) +
  f_C(t)\right](\omega_t-\omega^{(n)}_t)(dx)\, dt,\\
\mathcal E^R(\omega^{(n)},\eta^{(n)}) &= {\max_{\eta\in\overline{ \mathcal A}(\eta_0)}} \int_0^T
\int_{\overline\Omega} e^{-\rho t} \left[-P_t(\omega^{(n)}, \eta^{(n)})x +
  f_R(t)\right](\eta_t-\eta_t^{(n)})(dx)\, dt.
\end{align*}
{Clearly, $\mathcal E^C(\omega,\eta)\geq 0$ and $\mathcal
E^R(\omega,\eta)\geq 0$ for all $\omega \in \mathcal A(\omega_0)$ and $\eta\in
\overline{\mathcal A}(\eta_0)$, and the situation when $\mathcal E^C(\omega,\eta)=\mathcal
E^R(\omega,\eta)=0$ corresponds to the Nash equilibrium.} In general,
$\mathcal E^C(\omega,\eta)$ corresponds to the increase of gain of
all conventional producers if they move from the current value
$\omega$, to their best response, supposing that the distribution of
renewable producers remains unchanged, and similarly for the renewable
producers.

We do not prove the convergence of the algorithm, however,
in the illustrations presented in the next section we observe
convergence of $\mathcal E^C(\omega^{(n)},\eta^{(n)})$ and $\mathcal
E^R(\omega^{(n)},\eta^{(n)})$ to zero at the rate $\frac{1}{n}$, thus
the algorithm produces an $\varepsilon$-Nash equilibrium in
$O(\varepsilon^{-1})$ steps.

%In our implementation, we stop the algorithm as soon as
%increase of gain upon switching to the best response becomes inferior
%to 1 cent per hour per megawatt of electricity. This corresponds to
%roughly 500 iterations of the algorithm. 

\paragraph{Computing the best response}
To compute the best response numerically, we discretize the
Fokker-Planck inequalities in the definition of the sets $\mathcal A$
and $\overline{\mathcal A}$, in time and in space. The process $S$
(renewable output) is discretized on the interval $(S_{min},S_{max})$
with $S_{min}=0$ and $S_{max}=1$, using a uniform grid with $N_s+1$
points. The process $C$ (conventional cost) is similarly discretized
on the interval $(C_{min},C_{max})$ with $C_{min}=0$ and $C_{max}$
chosen depending on the model parameters, using a uniform grid with
$N_{C}+1$ points. Finally, the time interval is also discretized using a
uniform grid with $N_T+1$ points. We define $\Delta S=
\frac{S_{max}-S_{min}}{N_S}$, $S_j  = S_{min} + j \Delta S$ for
$j=0,\dots,N_S$, and similarly for the other variables. 
{Every measure $\omega\in \mathcal A(\omega_0)$}
satisfies, in the sense of distributions, the Fokker-Planck inequality
$$
-\frac{\partial \omega_t}{\partial t} -k\frac{\partial}{\partial x}
((\theta-x)\omega_t) + \frac{\delta^2}{2}\frac{\partial^2}{\partial x^2} (x\omega_t) \geq 0.
$$
This inequality is discretized using the implicit scheme, leading to
the following system of inequailties:
\begin{align*}
\omega^j_i \geq \left(1+ \frac{2\sigma^2_j \Delta t}{\Delta C^2}\right) \omega^j_{i+1}-
  {\Delta t}\left(\frac{\sigma_{j+1}^2}{\Delta C^2}-
  \frac{\mu_{j+1}}{2\Delta C}\right)\omega^{j+1}_{i+1} - {\Delta t}\left(\frac{\sigma_{j-1}^2}{\Delta C^2}+
  \frac{\mu_{j-1}}{2\Delta C}\right)\omega^{j-1}_{i+1} 
\end{align*}
for $0\leq j \leq N_C$ and $i=0,\dots,N_{T-1}$, where for $j=0$ and $J=N_C$, these formulas
are interpreted by assuming that $\omega^{-1}_i = \omega^{N_C}_i=0$
for $i=0,\dots,N_T$. 
 In the above formula, $\omega^j_i$ denotes the
discretized density at the point $(T_i,C_j)$, $\mu_j = k(\theta-C_j)$
and $\sigma^2_j = \frac{\delta^2 C_j }{2}$. The inequalities for
$\eta$ are discretized similarly. The gain functional is similarly
approximated by the discrete sum: for example, for the conventional
producers we have, 
$$
\sum_{i=0}^{N_T} \sum_{j=0}^{N_C} e^{-\rho T_i} [G(P_{T_i}(\omega_i,
\eta_i) - C_j) + f_C(T_i)]\omega^j_i. 
$$
The best response is then computed by maximizing this functional under
the inequality constraints given above and the positivity constraints,
using an interior point method for linear programming. 
\section{Illustration}
\label{illustration.sec}
The goal of this section is to provide a toy
example inspired by the British electricity sector to illustrate our model, rather than use it to obtain
realistic projections, which will be the topic of future research. Table \ref{ukgen.tab} shows
the UK installed generation capacity in 2017. The main energy sources are gas, nuclear
and intermittent renewables, which is 80\% wind. We therefore consider an economy
consisting of only these three sources of electricity. Our
model takes into account the departures of conventional producers and
entry of new renewable projects into the market. 
 The nuclear energy and the pumped storage generation is accounted for
 as baseline supply. The maximum baseline supply is therefore taken to
 be equal to 12.1 GW. We make an ad hoc assumption that the baseline supply increases linearly as function of price, from zero at zero price, to 12.1 at the maximal price value.

\begin{table}
\noindent\begin{tabular}{lllll}\hline
Conventional steam & CCGT &Nuclear & Pumped storage &Wind \& Solar\\ 
18.0 &  32.9 &  9.4 &   2.7  &    40.6 \\\hline 
\end{tabular}

\caption{UK Electricity installed generation capacity in 2017,
  GW. Conventional steam includes coal and gas. CCGT stands for
  combined cycle gas turbine. Wind and solar is
  approximately 20\% solar and 80\% wind, out of which there is about
  60\% onshore and 40\% offshore. Source: UK Energy in Brief 2018} 
\label{ukgen.tab}

\end{table}

The total installed gas capacity is taken to be equal to 35.9 GW.  The
total installed renewable capacity is taken to be equal to $35.6$ GW. 
To estimate the potential additional renewable generation capacity, we
  use the UK government projections\footnote{\texttt{\tiny https://www.gov.uk/government/publications/updated-energy-and-emissions-projections-2018}}. Under
  the high fossil fuel price scenario, the additional renewable
  generation capacity by 2035 is estimated at 55 GW, and under the
  reference scenario at 42 GW. We use the value $47$ GW  in
  the examples below. 

\paragraph{Capital costs, discount rate and depreciation rate for renewable
    producers.} The report \cite{thomson2015life} summarizes 17 studies
  of capital costs of onshore wind power plants. Taking the average of these 17
  values, we obtain a mean capital cost of 1377 GBP per kilowatt of
  energy (in 2011 GBP) and a mean discount rate of $8.6\%$. The
  annual operational and maintenance cost are estimated in this report
  to be between $1\%$ and $1.5\%$ of the capital costs, and we use the
  value of $1.25\%$ in the simulation. Finally, the lifetime of the
  wind power plants is taken to be equal to 20 years, and we take
  $\gamma_R = \frac{\log 2}{10}$ meaning that the plant looses $50\%$
  of its value over 10 years. 

\paragraph{Capital costs and depreciation rate for conventional
  producers}
We assume that the fixed running cost of the conventional power plants
is $k_C = 30$ GBP per MW of capacity per year (see
\cite{leigh2016electricity}). Upon exiting the market, the
conventional producer is assumed to lose the value of the plant in its
entirety ($K_C = 0$). 

%We consider the amount $K_C$  recovered by the conventional producer
%upon exiting the market to be a fraction of the capital cost of the
%power plant. To estimate this value we consider the capital cost of H Class CCGT,
%  based on a 1200 MW unit consisting of two 600 MW blocks, 
%  as estimated in \cite{leigh2016electricity}. The authors provide a low estimate of 439
%  GBP per MW (2014 prices), a medium estimate of 516 GBP and a high
%  estimate of 593 GBP. In this study we use the medium estimate of
%  $516$ GBP. Moreover, we use the depreciation rate $\gamma_C =
%  \frac{\log 2}{30}$, and we assume that the average age of the plants
%  in the beginning of the simualtion is 20 years. The running cost for
%  conventional plants is taken to be $k_C = 26.96 + 0.004 K_C$
% (also from \cite{leigh2016electricity}). 

\paragraph{Initial distribution and dynamics of the capacity factors}
To find a plausible initial distribution of renewable capacity factors, we
  computed the capacity factors of 20 largest UK onshore wind plants
  for the year 2018. The list of power plant locations was downloaded
  at\\
  \texttt{\footnotesize https://en.wikipedia.org/wiki/List\_of\_onshore\_wind\_farms\_in\_the\_United\_Kingdom},
  and the capacity factors were computed using the software at\\
  \texttt{https://www.renewables.ninja/}, which uses MERRA-2
  reanalysis dataset. The mean of the 20 values is $42.61\%$ and the
  standard deviation is $4.43\%$. We therefore calibrate the
  mean and variance of the stationary distribution of the capacity
  factor process to these values.

{The stationary distribution of the Jacobi process (see e.g. \cite{gourieroux}),}
also known as Wright-Fisher diffusion \eqref{jacobi}
(see \cite[Chapter 10]{ethier2009markov})
is the two-parameter beta distribution given by
$$
p(x) = \frac{ x^{\frac{2\bar k\bar\theta}{\bar\delta^2}-1}
  (1-x)^{\frac{2\bar k(1-\bar\theta)}{\bar\delta^2}-1}}{B\left(\frac{2\bar k\bar\theta}{\bar\delta^2},
  \frac{2\bar k(1-\bar\theta)}{\bar\delta^2}\right)},
$$
where $B$ is the beta function. The mean and variance of this
invariant distribution are $\bar \theta$ and $\frac{\bar\theta(1-\bar\theta)\bar
  \delta^2}{2\bar k+\bar\delta^2}$, respectively. To calibrate the
  parameters $(\bar k, \bar \theta, \bar \delta)$, we need a third
  constraint, which cannot be obtained from the stationary
  distribution. To this end, we fixed the parameter $\bar k$ in an ad
  hoc manner to $\bar k = 0.5$. 

\paragraph{Initial distribution and dynamics of the cost processes of
  gas-fired power plants}
It is difficult to quantify the full runinng cost of a gas-fired
power plant. To have an idea of the distribution of such costs, we study the aggregate
offer curves from the spot electricity market. For reasons of data
availability, we use the curves from the French electricity markets,
assuming that the costs are roughly the same in France and in the UK,
after accounting for exchange rate. The typical aggregate offer curve
in the French spot market, truncated to the price interval from -50
EUR to 100 EUR, is shown in Figure \ref{agrecurve.fig}, left graph. We see that
this curve can be split into different almost linear segments,
corresponding to different fuel types. {We fit the aggregate offer
curve using a piecewise-constant
function with four breakpoints} (shown as a thick solid line
in Figure \ref{agrecurve.fig}, left graph) and identify the longest
linear segment as corresponding to offers by gas-fired power
plants. This is motivated by the fact that gas has the largest share
of flexible generation (excluding nuclear) in France. This gives a
distribution of bids for a specific hour, which is then averaged over
24 hours to obtain the daily distribution. This analysis is performed
on each day for the period from January 1st, 2016 to October 5th,
2017. This gives us a cost distribution for each day of this reference
period.  Figure \ref{agrecurve.fig}, right graph, shows the evolution
of the mean and standard deviation of this distribution, together
with the evolution of the fuel price (spot gas price for trading
region France), converted to electricity price using an efficiency
factor of 44\%. The average values of the mean and standard deviation
are, respectively, $33.4$ and $11.0$ (after conversion to GBP). We
therefore calibrate the mean and standard deviation of the stationary
distribution of the cost factor
process to these values. 

{The stationary distribution of the CIR process \eqref{cir} (see e.g. \cite{kebab}) is the
two-parameter gamma distribution.}
$$
p(x) = \frac{\left(\frac{\delta^2}{2k}\right)^{-\frac{2k\theta}{\delta^2}}}{\Gamma\left(\frac{2k\theta}{\delta^2}\right)} x^{\frac{2k}{\delta^2}\theta-1} e^{-\frac{2k}{\delta^2} x}
$$
The mean and variance of this
invariant distribution are $\theta$ and $\frac{\theta
  \delta^2}{2 k}$, respectively. As before, to calibrate the
  parameters $( k, \theta, \delta)$, we need a third
  constraint, which cannot be obtained from the stationary
  distribution and we fixed the parameter $ k$ in an ad
  hoc manner to $ k = 0.5$.

\begin{figure}
\centerline{\includegraphics[width=0.5\textwidth]{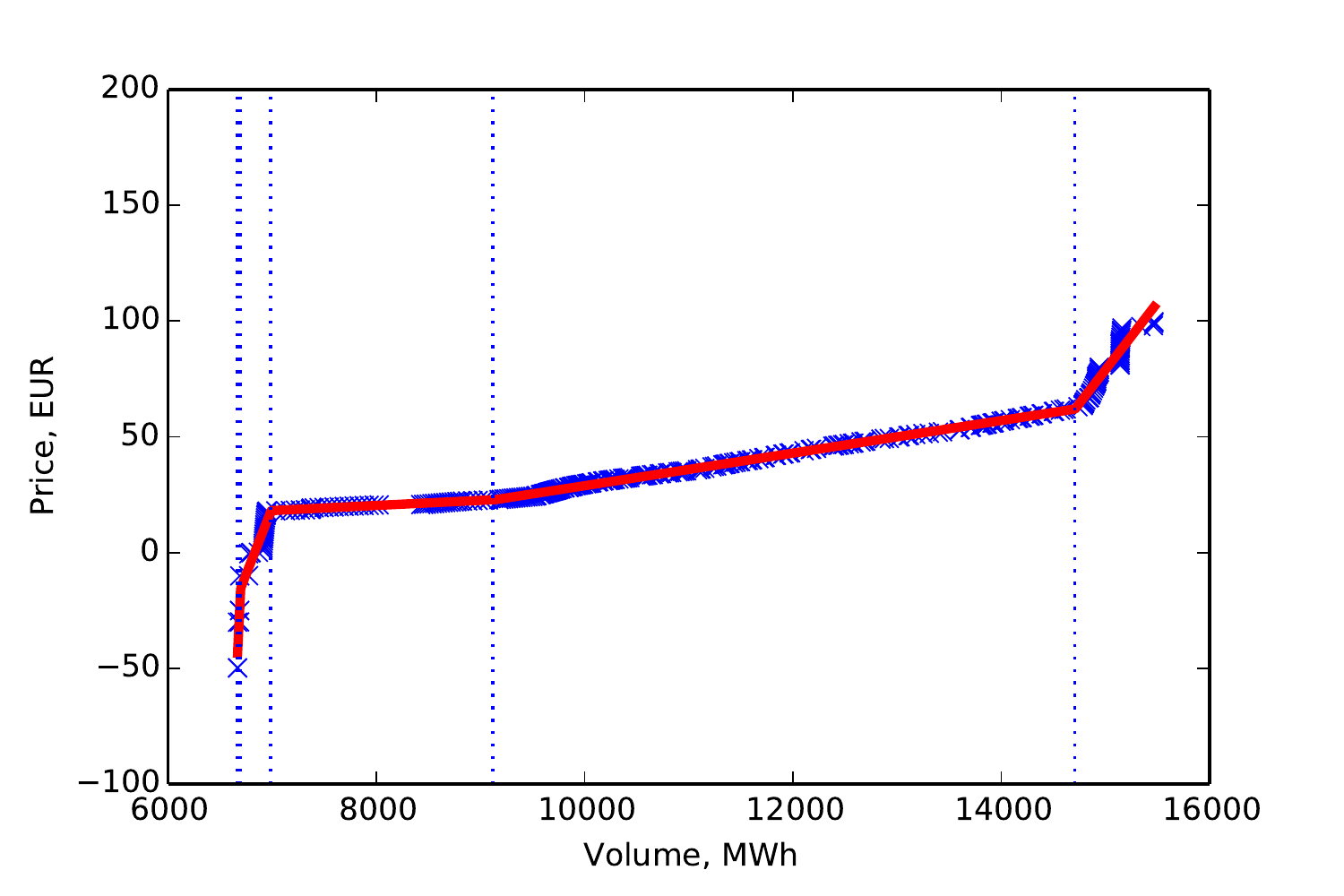}\includegraphics[width=0.5\textwidth]{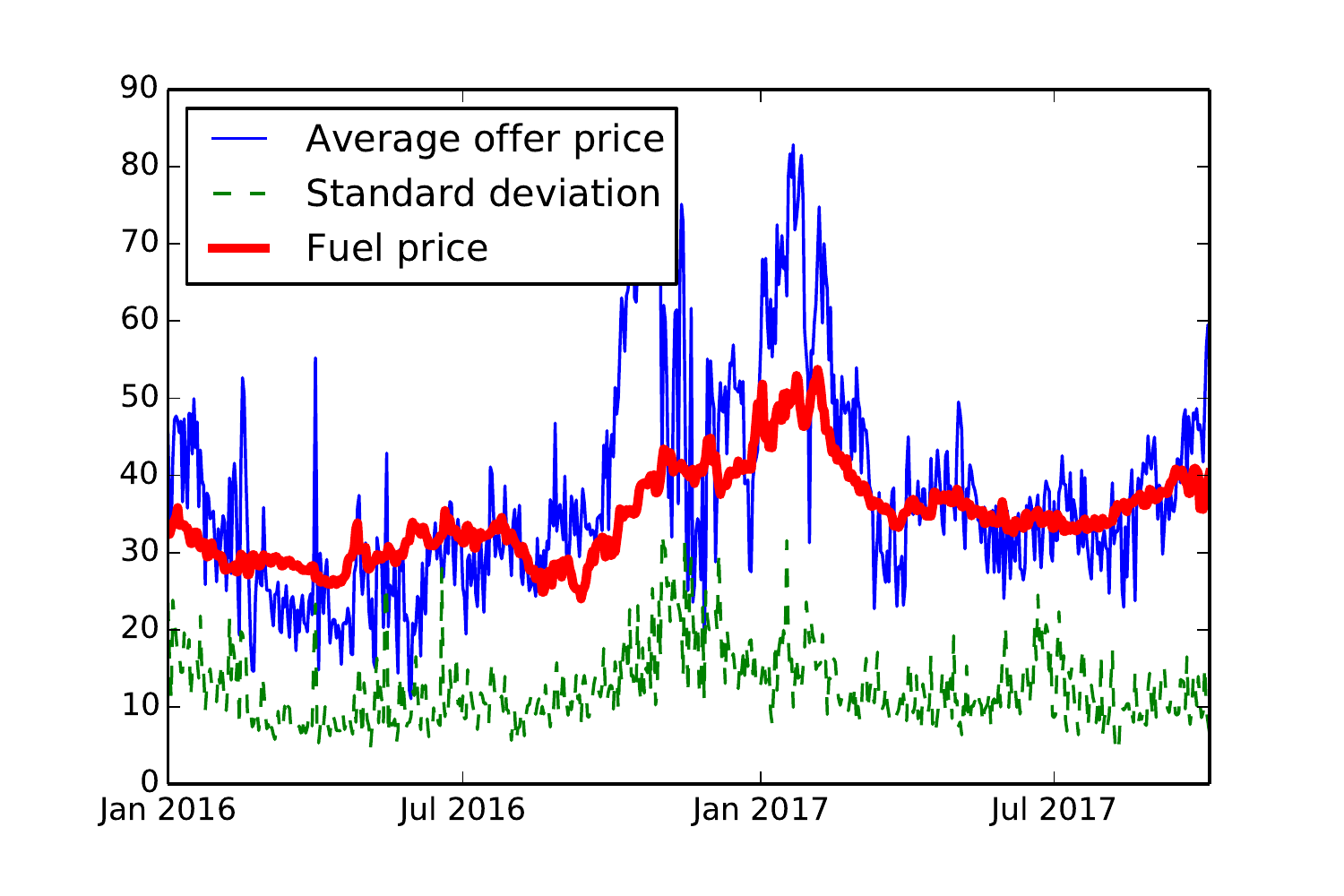}}
\caption{Left: typical offer curve in the French electricity
  market. Right: evolution of the daily mean and standard deviation of
the offers by gas-fired power plants in the French electricity market.}
\label{agrecurve.fig}
\end{figure}

\paragraph{Bidding function of gas-fired power plants} Since it is
  difficult to reconstruct the actual bidding function used by market
  participants, we use an ad hoc bidding function, satisfying our
  assumptions, which takes the following form:
$$
F(x) = \mathbf 1_{x>\epsilon}+\frac{1}{2}(1+\sin(-\pi/2+\pi
x/\epsilon))\mathbf 1_{0\leq x \leq \epsilon}, 
$$
where $\epsilon$ is a parameter which we choose equal to $0.5$. This
means that the conventional generators bid their full capacity as soon
as the price is greater than the cost plus $0.5$ GBP. The actual
choice of $\epsilon$ does not have a significant effect on the
results. 
\paragraph{Electricity demand projections.} The reference electricity consumption evolution scenario
  is taken from the British government projections\footnote{\texttt{\tiny
    https://www.gov.uk/government/publications/updated-energy-and-emissions-projections-2018},
annex F}. This reference provides a forecast of average annual
electricity consumption up to 2035. In our model, the time step was
fixed to 3 months and a distinction between peak (7AM-8PM, Mon-Fri)
and off-peak demand has been made. To this end, we have used the high
frequency electricity consumption data from
\texttt{gridwatch.co.uk} to estimate the historical annual cycle and the
historical peak/off-peak ratio. This enabled us to construct the
projections of peak and off-peak consumption with the time step of 3
months. These are shown in Figure \ref{demandproj.fig}.

\paragraph{Distinction between peak and off-peak prices}
For realistic modeling of electricity markets, it is essential to
distinguish between peakload and baseload (off-peak) prices. In particular, the
effect of renewable penetration on these two prices may be
different. To this end, the simulations were carried out in a slightly
modified version of the model, where peak and off-peak prices are
computed separately, by matching the corresponding demand projection
with the supply curve of the market, and the revenues of the agents
are computed by adding up their revenues over peak and off-peak
periods. All the theoretical developments carry over to this case, but
we chose to present the single-price case in the paper to lighten the
notation.

\begin{figure}
\centerline{\includegraphics[width=0.8\textwidth]{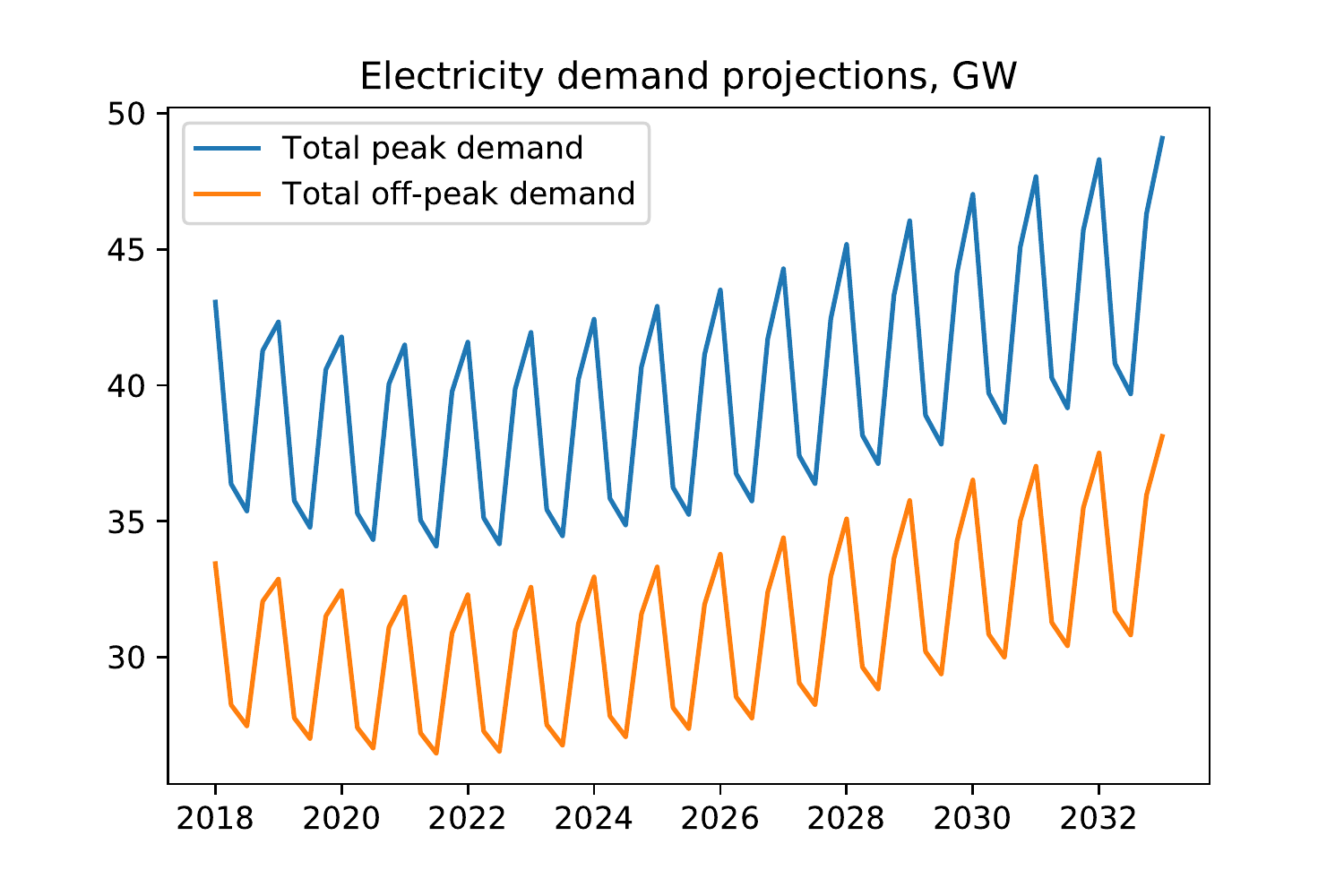}}
\caption{Electricity demand projections used in the simulation.}
\label{demandproj.fig}
\end{figure}

\paragraph{Simulation results}
The first simulation illustrates the baseline case (with the parameters
and modeling choices described above). We perform 300 iterations of
the algorithm as described in Section \ref{numerics.sec}, and monitor the gain
increase from switching to the best response for renewable and
conventional producers $\mathcal E^C(\omega^{(n)},\eta^{(n)})$ and $\mathcal E^R(\omega^{(n)},\eta^{(n)})$. Figure \ref{conv.fig} shows the gain increase
for the renewable and conventional producers in the baseline
simulation: it can be seen that this
quantity converges to zero at the rate $\frac{1}{n}$ where $n$ is the
number of iterations. The final values for the baseline simulations
correspond to a gain increase of about $0.001$ GBP per MW of installed
capacity per hour for the conventional producers and $0.005$ GBP per MW of installed
capacity per hour for the renewable producers, which is quite small compared
to the price at which electricity is usually sold. Similar convergence
rates and similar or smaller final values have been observed in other
simulations described below.

\begin{figure}
\centerline{\includegraphics[width=0.8\textwidth]{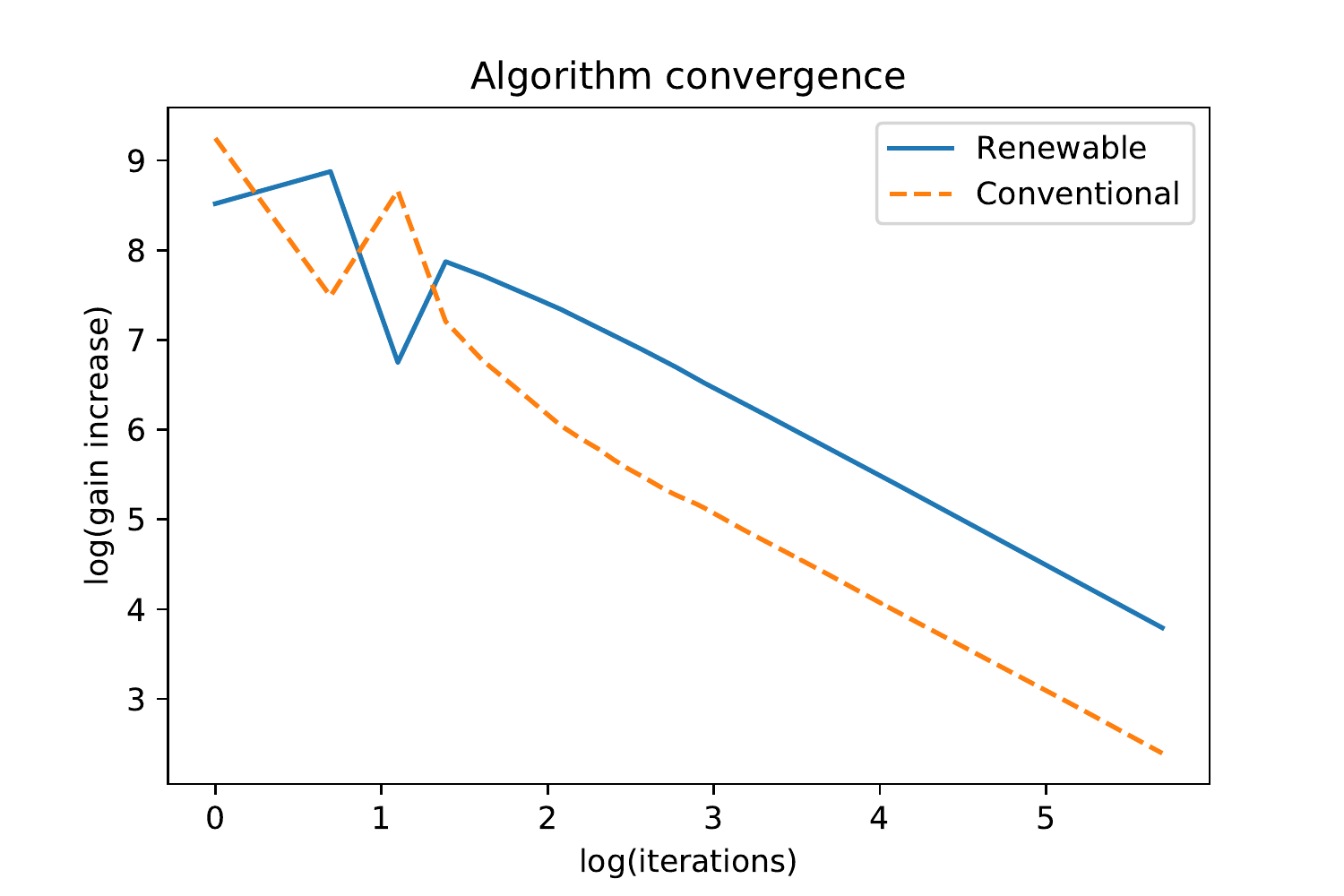}}
\caption{Convergence of the gain increase upon switching to the best
  response. The gain is given in million GBP for the entire sector.}
\label{conv.fig}
\end{figure}

In the second simulation, which we term Scenario 1, we assume that there is a 27\% subsidy
for renewable energy, bringing down the fixed cost of building a wind
power plant to $1000$ GBP per kW of installed capacity. The other
parameters are same as above.

In the third simulation, termed Scenario 2, we assume that the
conventional producers receive a fixed payment of 10 GBP per KW of
installed capacity per year. For comparison, the clearing prices at
the UK T-4 capacity auctions were $19.4$ GBP/kW for 2018--2019,  $18$
GBP/kW for 2019--2020, $22.5$ GBP/kW for 2020--2021 and  $8.4$ GBP/kW
for 2021-2022, see \cite{capacityreport}.  The 27\% renewable subsidy is still in
place. 

Figure \ref{convren.fig} shows the evolution of the conventional and
renewable installed capacity in the three simulations. The
corresponding price trajectories are shown in Figure
\ref{price.fig}. While in the baseline scenario, no new renewable
capacity is installed, the 27\% renewable subsidy (scenario 1) dramatically
increases renewable installation, which practically doubles over the
15-year period. Notice that while the conventional capacity is reduced
in the beginning of the 15-year period, the arrival of the renewable
capacity is more gradual. This happens because of the form of the
demand process (Figure \ref{demandproj.fig}) which grows in the second
half of the 15-year period. 

In terms of electricity prices, the main effect of the
renewable subsidy is a reduction of the baseload price, especially
during the winter months. On the other hand, the peakload electricity
price in winter months increases compared to the baseline scenario. Thus, the renewable subsidy
 allows to decarbonize the electricity production, but may lead to
 higher peakload prices. This happens in our model because the
 conventional producers are strongly incentivised to leave the
 market. In scenario 2 this incentive is reduced (in our case this is
 achieved with capacity payments). As a result,
 the renewable penetration is slightly reduced, but the peakload price
 is considerably lower than in scenario 1, and also lower
 than in the baseline scenario. The baseload price, on the other hand,
 is much lower than in the baseline scenario, slightly higher than in
 scenario 1 in the summer months and similar to scenario 1 in the
 winter months. 

\begin{figure}
\centerline{\includegraphics[width=0.5\textwidth]{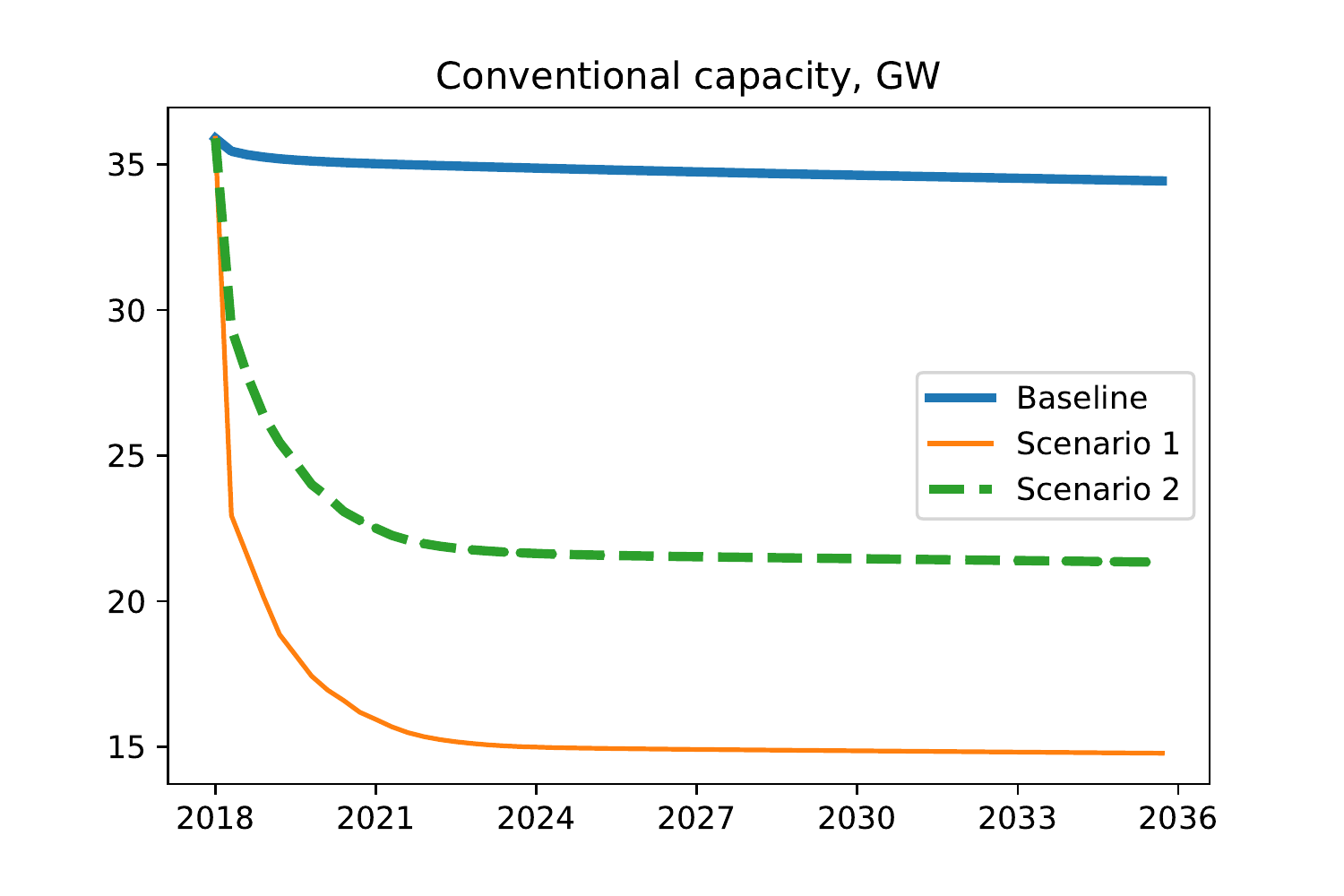}\includegraphics[width=0.5\textwidth]{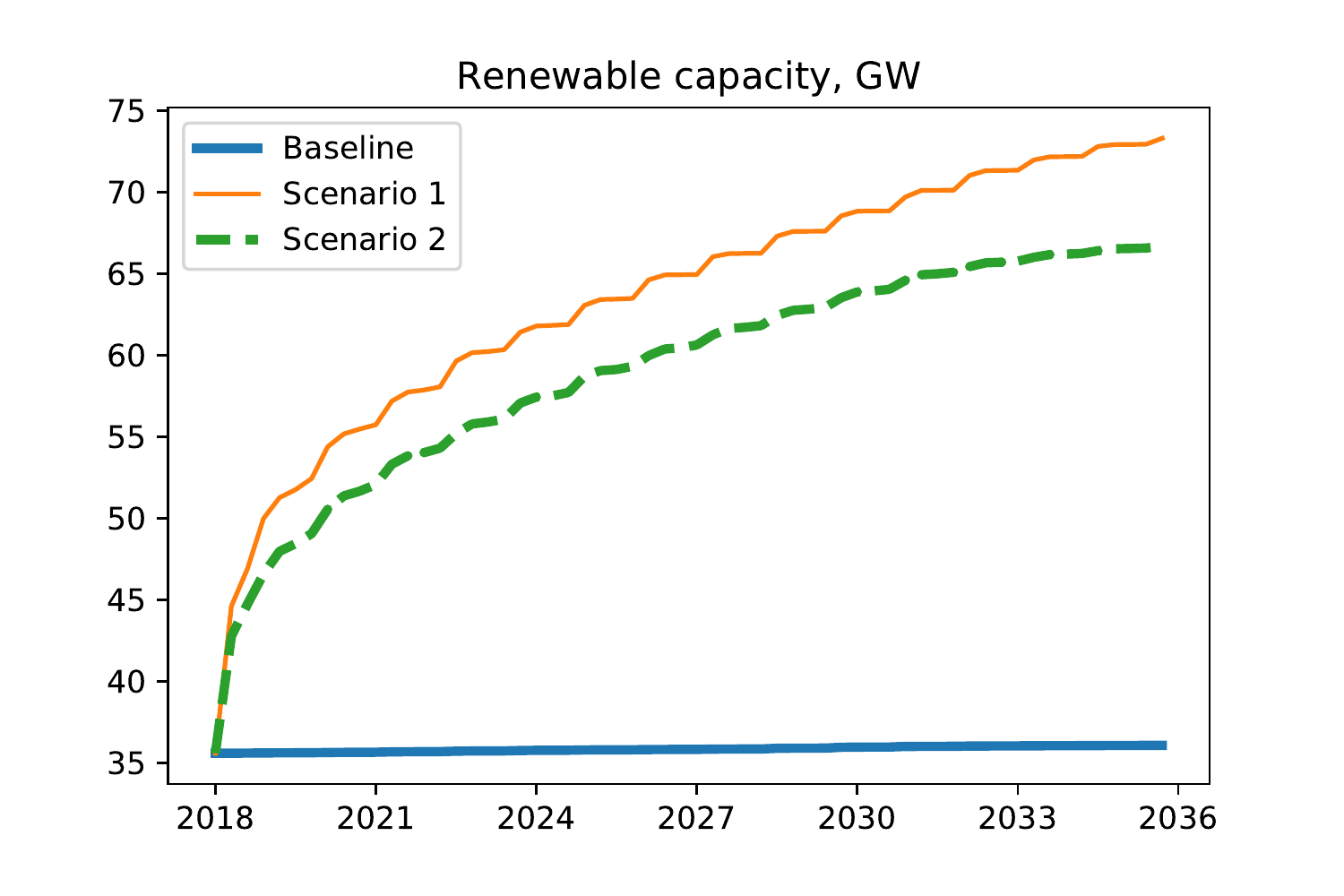}}
\caption{Evolution of conventional and renewable installed capacity in
  the three simulations. Scenario 1: renewable subsidy. Scenario 2:
  renewable subsidy and capacity payments for conventional plants.}
\label{convren.fig}
\end{figure}

\begin{figure}
\centerline{\includegraphics[width=\textwidth]{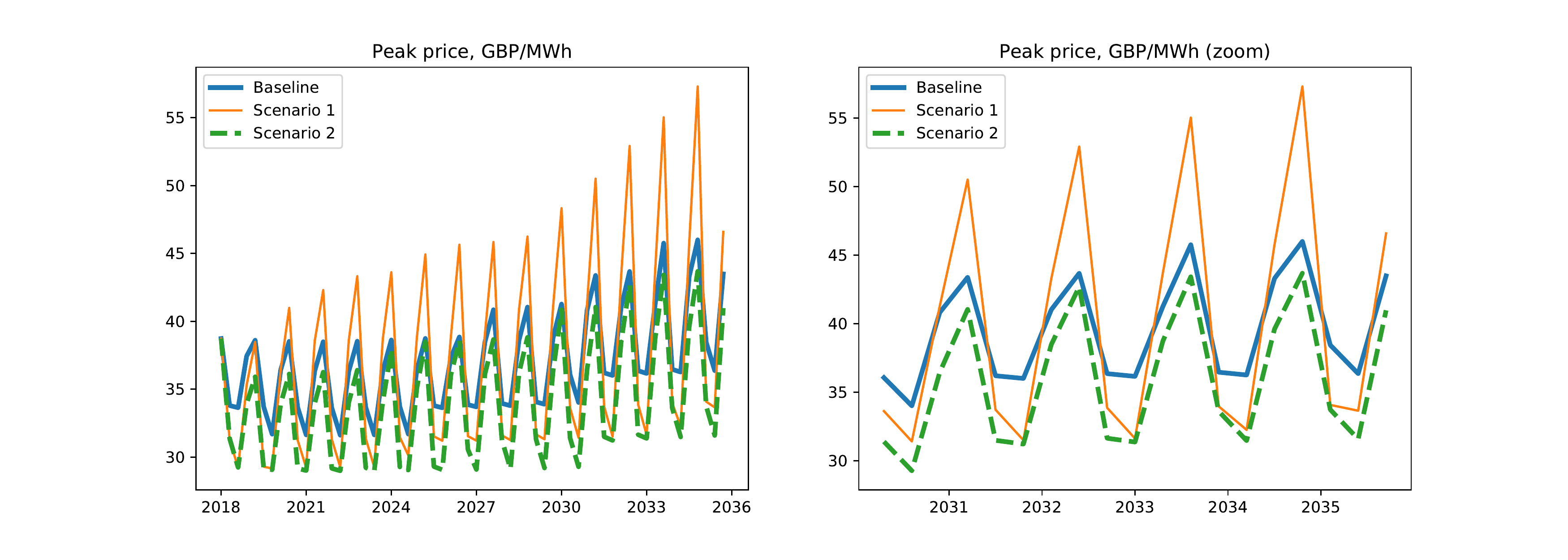}}
\centerline{\includegraphics[width=0.7\textwidth]{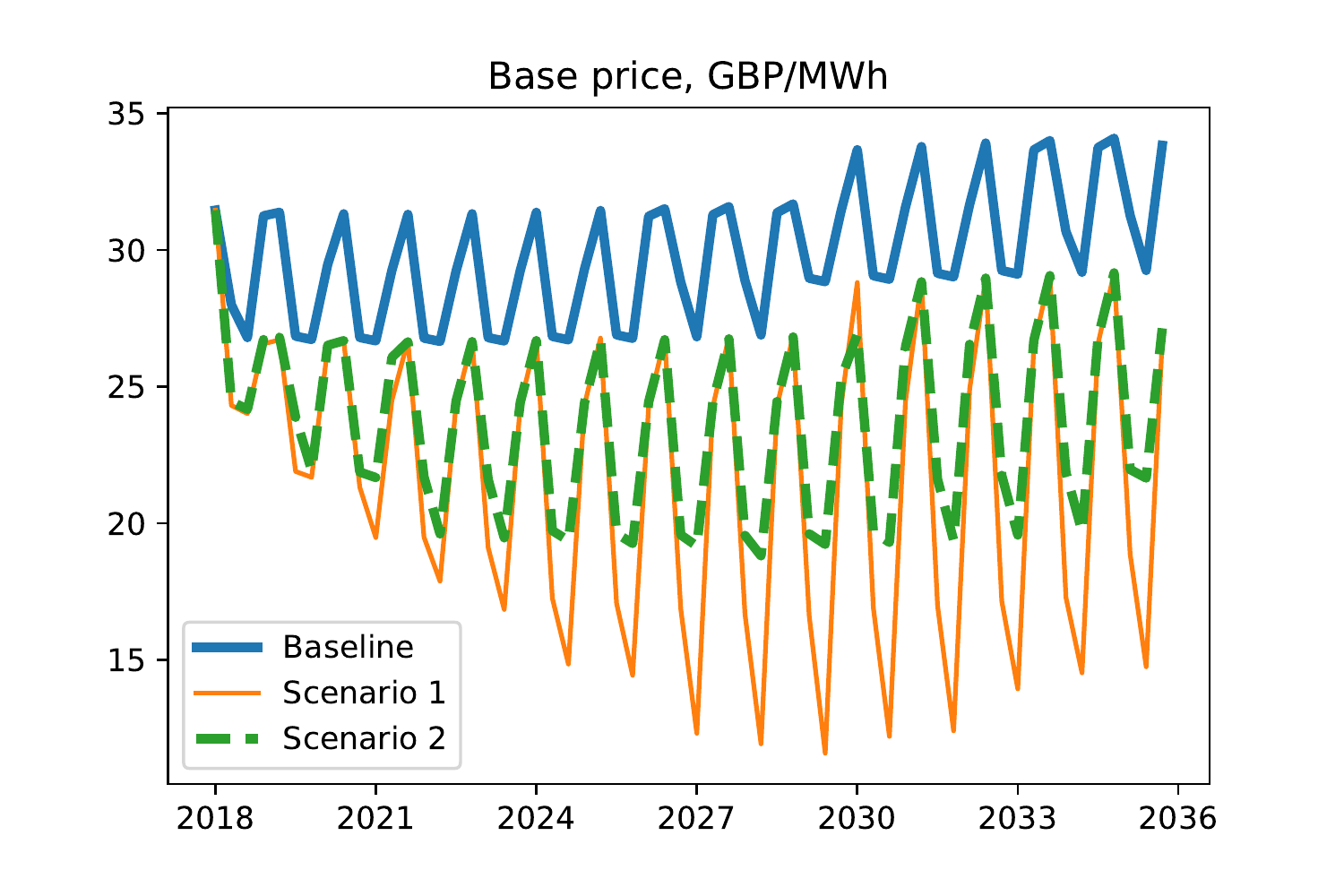}}
\caption{Evolution of the electricity price in
  the three simulations. Top left: peak price. Top right: peak price
  (zoom on the last 4 years). Bottom: base price. 
Scenario 1: renewable subsidy. Scenario 2:
  renewable subsidy and capacity payments for conventional plants.}
\label{price.fig}
\end{figure}

We conclude that while renewable subsidies clearly lead to higher
renewable penetration, this may entail a cost to the consumer in terms
of higher peakload prices. In order to avoid rising prices, the
renewable subsidies must be combined with market or off-market
mechanisms ensuring that sufficient conventional capacity remains in
place to meet the energy demand during peak periods. 

\section{Acknowledgement}
Peter Tankov and Ren\'e A\"id gratefully acknowledge financial support from the ANR
(project EcoREES ANR-19-CE05-0042) and from the
FIME Research Initiative.

%\bibliographystyle{siam}
%\bibliography{renewable}

\begin{thebibliography}{10}

\bibitem{ALL17}
{\sc R.~A{\"\i}d, L.~Li, and M.~Ludkovski}, {\em Capacity expansion games with
  application to competition in power generation investments}, Journal of
  Economic Dynamics and Control, 84 (2017), pp.~1--31.

\bibitem{AMT18}
{\sc C.~Alasseur, I.~Ben~Tahar, and A.~Matoussi}, {\em An extended mean field
  game for storage in smart grids}, arXiv preprint arXiv:1710.08991,  (2017).

\bibitem{ambrosio2000functions}
{\sc L.~Ambrosio, N.~Fusco, and D.~Pallara}, {\em Functions of bounded
  variation and free discontinuity problems}, vol.~254, Clarendon Press Oxford,
  2000.

\bibitem{BB14}
{\sc F.~Bagagiolo and D.~Bauso}, {\em Mean-field games and dynamic demand
  management in power grids}, Dynamic Games and Applications, 4 (2014),
  pp.~155--176.

\bibitem{kebab}
{\sc M.~Ben~Alaya and A.~Kebaier}, {\em Parameter estimation for the square
  root diffusions: ergodic and nonergodic cases}, Stochastic Models, 28, Iss.4
  (2012), pp.~609--634.

\bibitem{Ben18}
{\sc D.~Benatia}, {\em Functional econometrics of multi-unit auctions: An
  application to the {N}ew {Y}ork electricity market},  (2018).
\newblock Working paper.

\bibitem{B18}
{\sc C.~Bertucci}, {\em Optimal stopping in mean field games, an obstacle
  problem approach}, Journal de Math{\'e}matiques Pures et Appliqu{\'e}es, 120
  (2018), pp.~165--194.

\bibitem{BMR17}
{\sc P.~C. Bhagwat, A.~Marcheselli, J.~C. Richstein, E.~J. Chappin, and L.~J.
  De~Vries}, {\em An analysis of a forward capacity market with long-term
  contracts}, Energy policy, 111 (2017), pp.~255--267.

\bibitem{bouveret2018mean}
{\sc G.~Bouveret, R.~Dumitrescu, and P.~Tankov}, {\em Mean-field games of
  optimal stopping: a relaxed solution approach}, SIAM Journal on Control and
  Optimization,  (to appear).

\bibitem{BLB18}
{\sc C.~Byers, T.~Levin, and A.~Botterud}, {\em Capacity market design and
  renewable energy: Performance incentives, qualifying capacity, and demand
  curves}, The Electricity Journal, 31 (2018), pp.~65--74.

\bibitem{CCZ15}
{\sc S.~Cl{\`o}, A.~Cataldi, and P.~Zoppoli}, {\em The merit-order effect in
  the {I}talian power market: The impact of solar and wind generation on
  national wholesale electricity prices}, Energy Policy, 77 (2015), pp.~79--88.

\bibitem{CPTD12}
{\sc R.~Couillet, S.~M. Perlaza, H.~Tembine, and M.~Debbah}, {\em A mean field
  game analysis of electric vehicles in the smart grid}, in 2012 Proceedings
  IEEE INFOCOM Workshops, IEEE, 2012, pp.~79--84.

\bibitem{ethier2009markov}
{\sc S.~N. Ethier and T.~G. Kurtz}, {\em Markov processes: characterization and
  convergence}, vol.~282, John Wiley \& Sons, 2009.

\bibitem{Fab18}
{\sc N.~Fabra}, {\em A primer on capacity mechanisms}, Energy Economics, 75
  (2018), pp.~323--335.

\bibitem{gomes2015obstacle}
{\sc D.~Gomes and S.~Patrizi}, {\em Obstacle mean-field game problem},
  Interfaces and Free Boundaries, 17 (2015), pp.~55--68.

\bibitem{gomes2014mean}
{\sc D.~Gomes and J.~Sa\'ude}, {\em Mean field games models -- a brief survey},
  Dynamic Games and Applications, 4 (2014), pp.~110--154.

\bibitem{gomes2020mean}
\leavevmode\vrule height 2pt depth -1.6pt width 23pt, {\em A mean-field game
  approach to price formation}, Dynamic Games and Applications,  (2020),
  pp.~1--25.

\bibitem{gourieroux}
{\sc C.~Gouri\'eroux and P.~Val\'ery}, {\em Estimation of a {J}acobi process},
  Preprint,  (2004).

\bibitem{HG13}
{\sc A.~Henriot and J.-M. Glachant}, {\em Melting-pots and salad bowls: The
  current debate on electricity market design for integration of intermittent
  res}, Utilities Policy, 27 (2013), pp.~57--64.

\bibitem{IAE2017}
 {\sc International Energy Agency}, {\em Energy Technology Prospectives
   Report}, June 2017. 

\bibitem{leigh2016electricity}
{\sc Leigh~Fisher~Jacobs}, {\em Electricity generation costs and hurdle rates},
  tech. rep., Department of Energy and Climate Change, 2016.

\bibitem{KP17}
{\sc R.~Kiesel and F.~Paraschiv}, {\em Econometric analysis of 15-minute
  intraday electricity prices}, Energy Economics, 64 (2017), pp.~77--90.

\bibitem{LL07}
{\sc J.-M. Lasry and P.-L. Lions}, {\em Mean field games}, Japanese journal of
  mathematics, 2 (2007), pp.~229--260.

\bibitem{LB15}
{\sc T.~Levin and A.~Botterud}, {\em Electricity market design for generator
  revenue sufficiency with increased variable generation}, Energy Policy, 87
  (2015), pp.~392--406.

\bibitem{forbes}
{\sc B.~Murray}, {\em The paradox of declining renewable costs and rising
  electricity prices}, Forbes,  (2019).

\bibitem{capacityreport}
{\sc Department~of~Business~Energy and Industrial~Strategy}, {\em Capacity market five-year
  review 2014--2019}, tech. rep., UK Government, 2019.

\bibitem{RPR15}
{\sc V.~Rious, Y.~Perez, and F.~Roques}, {\em Which electricity market design
  to encourage the development of demand response?}, Economic Analysis and
  Policy, 48 (2015), pp.~128--138.

\bibitem{Sch15}
{\sc S.~Schwenen}, {\em Strategic bidding in multi-unit auctions with capacity
  constrained bidders: the {N}ew {Y}ork capacity market}, The RAND Journal of
  Economics, 46 (2015), pp.~730--750.

\bibitem{TGKM08}
{\sc R.~Takashima, M.~Goto, H.~Kimura, and H.~Madarame}, {\em Entry into the
  electricity market: Uncertainty, competition, and mothballing options},
  Energy Economics, 30 (2008), pp.~1809--1830.

\bibitem{thomson2015life}
{\sc R.~C. Thomson and G.~P. Harrison}, {\em Life cycle costs and carbon
  emissions of onshore wind power}, tech. rep., University of Edinburgh, 2015.

\bibitem{WV08}
{\sc A.~Weidlich and D.~Veit}, {\em A critical survey of agent-based wholesale
  electricity market models}, Energy Economics, 30 (2008), pp.~1728--1759.

\end{thebibliography}

\appendix
\section{Technical proofs}
{We provide here the complete proof of Lemma \ref{bvariation}.}\\

\textbf{Part i. Step 1. }
Introduce the function
$$
L_t = \left(\overline D_t - \int_{\overline\Omega}
{x(\bar \eta_t(dx)-\eta_t(dx))}\right)_+.
$$
The price process satisfies
$$
P_t(\omega_t,\eta_t) = \inf\{y\in [0,\overline P]: \int_\Omega F(y-x)
\omega_t(dx)  + F_0(y)\geq L_t\}\wedge \overline P.
$$
We would like to show that $L$ has bounded variation on $[0,T]$. 
To this end, let $\psi:[0,T]\mapsto \mathbb R$ be a $C^1$
function. {We first show that for all
$t\in [0,T]$, 
$$
\left|\mathbb E\left[\int_t^T \psi'(s) S^{t,x}_s ds\right] \right|\leq C \|\psi\|_\infty.
$$
for some constant $C<\infty$.
Indeed, by using It\^o's formula, we get:
$$ 
\mathbb{E} \left[\int_t^T \psi'(s)S_s^{t,x}ds  \right]=\mathbb{E} \left[S_T^{s,x}\psi(T)-S_t^{s,x}\psi(t) \right]-\mathbb{E} \left[\int_t^T \psi(s)\bar k(\bar \theta-S_s^{t,x})ds\right].
$$
Then, by using that the process $(S_t^{s,x})$ takes values between $0$ and $1$, the result easily follows.}
 Now, we can consider the test function
$$
u(t,x) = \mathbb E\left[\int_t^T \psi'(s) S^{t,x}_s ds\right]  + C\|\psi\|_\infty.
$$
{First note that, by construction, $u \geq 0$. Moreover, it is easy to check through explicit computation that
$u\in C^{1,2}([0,T]\times \overline \Omega)$.  Moreover, $\overline{\Omega}$ is bounded and consequently, $\frac{\partial
  u}{\partial t} + \overline{\mathcal L} u =  \psi'(t) x$ is bounded}
on $[0,T]\times \overline{\Omega}$. Plugging this
test function into the definition of $\overline{\mathcal A}(\eta_0)$ yields
$$
\int_0^T \psi'(s) \int_{\overline \Omega} x \eta_t(dx) \leq 2C \|\psi\|_\infty,
$$
which means that the total variation of the mapping
$$
t\mapsto \int_{\overline\Omega }x\eta_t(dx),
$$ 
is bounded on $[0,T]$ by $2C$, and thus also $L_t$, has bounded 
variation on $[0,T]$.\\

\textbf{Step 2. }
For a fixed $n$, define
$$
F^n_k(t):= \int_0^{\overline P} F(k\overline P/n-x)
\omega_t (dx) + F_0(k\overline P/n),\quad k=0,\dots,n. 
$$
We now would like to show that these functions have bounded variation
on $[0,T]$,
uniformly on $n,k$, in the sense that for every $n$ and every sequence
of $C^1$ functions
$\psi_k:[0,T]\mapsto \mathbb R$, $k=0,\dots,n$, 
\begin{align}
\sum_{k=0}^n \int_0^T F^n_k(t) \psi'_k(t) dt\leq C \max_{0\leq t\leq
  T} \sum_{k=0}^n |\psi_k |, \label{bvarfn}
\end{align}
where the constant $C$ does not
depend on $k,n$. To this end, as above, we start with the following
estimate, obtained by It\^o formula and integration by parts, where to
save space we omit the superscript $(t,x)$ of the process $C$. 
\begin{align*}
 & \sum_{k=0}^n\mathbb E\left[\int_t^T\psi'_k(s) F(k\overline
  P/n - C_s) ds\right] \\
& = \sum_{k=0}^n\mathbb E\Big[\int_t^T\psi'_k(s) \Big\{F(k\overline
  P/n - x) - \int_t^s F'(k\overline P/n - C_r)k(\theta - C_r) dr \\
 &\qquad +
  \frac{1}{2}\int_t^s F^{\prime\prime}(k\overline P/n -
  C_r)\delta^2 C_r dr\Big\} ds\Big]\\
& = \sum_{k=0}^n(\psi_k(T)-\psi_k(t))F(k\overline
  P/n - x) \\ &+\sum_{k=0}^n\mathbb E\Big[\int_t^T(\psi_k(T)-\psi_k(s)) \Big\{- F'(k\overline P/n - C_s)k(\theta - C_s)  +
  \frac{1}{2} F^{\prime\prime}(k\overline P/n -
  C_s)\delta^2 C_s\Big\} ds\Big]\\
&\leq C \max_{0\leq t \leq T} \sum_{k=0}^n \left|\psi_k(t)\right|,
\end{align*}
because $F(k\overline P/n-x)$ and the terms in curly brackets in the
last line above are bounded by a constant independent from $k$, $n$
and $x$ in view of the properties of $F$. 

Now let us consider the test function
$$
u(t,x) = \sum_{k=0}^n \mathbb E\left[\int_t^T\psi'_k(s) F(k\overline
  P/n - C^{(t,x)}_s) ds\right] + C \max_{0\leq t \leq T} 
\sum_{k=0}^n|\psi_k(t)|. 
$$
It is easy to check that it possesses the required regularity
properties, and satisfies
{$$
\frac{\partial u}{\partial t}(t,x)+\mathcal L u(t,x) = \sum_{k=0}^n\psi'_k(t) F(k\overline P/n - x). 
$$}

{Due to the boundedness property of the map $F$, we get that $\frac{\partial u}{\partial t}+\mathcal L u$ is bounded. Then, we plug it into the definition of $\mathcal A(\omega_0)$ and since 
$\int_0^{\bar{P}} F(k{\overline P}/n - x) \omega_t(dx)=\int_{\Omega} F(k\overline{P}/n - x) \omega_t(dx)$ for all $t \in [0,T]$, the result follows.}\\

\textbf{Step 3.} Let $\rho$ be a mollifier supported on $[-1,1]$, set $\rho_m(x):= m
\rho(mx)$ and define 
$$
F^{n,m}_k(t):= \rho_m * F^{n}_k (t), 
$$
where $F^{n}_k$ is extended by zero value outside the interval
$[0,T]$, so that $F^{n,m}_k$ is well defined on $[0,T]$. Let $\psi_0,\dots,\psi_n$ be a
sequence of test functions. Then, for every $m$ and
for $k=0,\dots,n$, 

{\begin{align*}
\sum_{k=0}^n \int_0^T F^{n,m}_k(t) \psi_k'(t) dt &= \sum_{k=0}^n \int_0^T
  \psi_k'(t) \int_{-1/m}^{1/m}  F^{n}_k(t-s) \rho_m(s)ds\, dt\\
& = \sum_{k=0}^n \int_{-1/m}^{1/m} \rho_m(s) ds \int_0^T  \psi_k'(t)
  F^{n}_k(t-s) dt\\
& = \sum_{k=0}^n \int_{-1/m}^{1/m} \rho_m(s) ds \int_{0}^{T}  \psi_k'(t+s)
  F^{n}_k(t) dt\\
&\leq C \max_{0\leq t\leq
  T}\sum_{k=0}^n |\psi_k(t)| \int_{-1/m}^{1/m} \rho_m(s) ds= C \max_{0\leq t\leq
  T}\sum_{k=0}^n |\psi_k(t)|,
\end{align*}
where we extend $\psi_k$ by constants outside the interval $[0,T]$. The last inequality follows by Step 2.}
By integration by parts, then,
\begin{align*}
\sum_{k=0}^n \int_0^T \frac{d}{dt} F^{n,m}_k(t) \psi_k(t) dt\leq C \max_{0\leq t\leq
  T}\sum_{k=0}^n |\psi_k(t)|,
\end{align*}
for a different constant $C$. Finally, by an approximation argument using the dominated convergence, this
implies that 

{\begin{align*}
\int_0^T \max_{0 \leq k \leq n} \left|\frac{d}{dt} F^{n,m}_k(t)\right| dt\leq C.
\end{align*}}

\textbf{Step 4.} By definition of $F_0$,
$F^{n}_{k+1}(t) - F^{n}_{k}(t)\geq \frac{c\overline P}{n}$, $0\leq k
\leq n-1$.
Now for $y\geq 0$ and $0= x_0<x_1<\dots <x_n$, define the mapping
$\Theta_n$ as follows. 
$$
\Theta_n(y,x_0,\dots,x_n) = \frac{\overline P}{n}
\sum_{k=0}^{n-1}\frac{(y-x_k)_+}{x_{k+1}-x_k}\wedge 1.
$$
Then,
\begin{align*}
\frac{\partial \Theta_n}{\partial y} &= \frac{\overline P}{n}
  \sum_{k=0}^{n-1}\frac{\mathbf 1_{x_k \leq y <x_{k+1}}
  }{x_{k+1}-x_k},\\ \frac{\partial \Theta_n}{\partial x_j} &= \frac{\overline P}{n}
  \frac{\mathbf 1_{x_j \leq y <x_{j+1}}}{(x_{j+1}-x_j)^2} (y-x_{j+1}) 
 - \frac{\overline P}{n}
  \frac{\mathbf 1_{x_{j-1} \leq y <x_{j}}}{(x_{j}-x_{j-1})^2}  (y-x_{j-1}).
\end{align*}
Introduce the discretized price
$$
P^n_t(\omega_t,\eta_t):= \Theta_n(L_t, F^n_0(t),\dots,F^n_n(t)). 
$$
Let $(L^m_t)_{m\geq 1}$ be a sequence of
$C^\infty$ functions approximating
$L_t$ in the sense of Theorem 3.9 in
\cite{ambrosio2000functions}, {and let $F^{n,m}_0,\dots,F^{n,m}_n$ be
defined as in Step 3}. Let $\phi: [0,T]\mapsto \mathbb
R$ be a $C^1$ function. Then,{
\begin{align*}
\int_0^T \phi'(t) P^n_t(\omega_t,\eta_t)dt &= \lim_{m\to \infty} \int_0^T \phi'(t)
\Theta_n(L^m_t, F^{n,m}_0(t),\dots,F^{n,m}_n(t))dt \\ &\leq
2\|\phi\|_\infty\overline P  + \lim_{m\to \infty}
\int_0^T \phi(t) \left\{\frac{\partial \Theta_n}{\partial y}
\frac{d}{dt}L^{m}_t + \sum_{j=0}^n \frac{\partial \Theta_n}{\partial
  x_j} \frac{d}{dt} F^{n,m}_j(t)\right\}dt\\
&\leq 2\|\phi\|_\infty\overline P + \|\phi\|_\infty \|\frac{\partial
  \Theta_n}{\partial y}\|_\infty \lim_m\int_0^T |\frac{d}{dt}L^{m}_t|
  dt \\ &+ \|\phi\|_\infty \lim_m \int_0^T \sum_k |\frac{\partial
  \Theta_n}{\partial x_k}|\max_k |\frac{d}{dt} F^{n,m}_k(t)|dt.
\end{align*}
Since $\frac{\partial
  \Theta_n}{\partial x_k}(t) \neq 0$ for only two $k$, we deduce that $|\frac{\partial
  \Theta_n}{\partial x_k}(t)| \leq 2c$ and therefore, 
\begin{align}\label{boundedvar}
\int_0^T \phi'(t) P^n_t(\omega_t,\eta_t)dt &\leq 2\|\phi\|_\infty\overline P + \|\phi\|_\infty \|\frac{\partial
  \Theta_n}{\partial y}\|_\infty \lim_m\int_0^T |\frac{d}{dt}L^{m}_t|
  dt \\ &+ \|\phi\|_\infty  \|\sum_k \frac{\partial
  \Theta_n}{\partial x_k}\|_\infty \lim_m \int_0^T \max_k |\frac{d}{dt} F^{n,m}_k(t)|dt \nonumber \\ &\leq 2\|\phi\|_\infty\overline P + c\|\phi\|_\infty \lim_m\int_0^T |\frac{d}{dt}L^{m}_t|
  dt \nonumber \\ &+ 2c \|\phi\|_\infty {\lim_m \int_0^T \max_k
         |\frac{d}{dt} F^{n,m}_k(t)|dt}\leq C \|\phi\|_\infty. \nonumber 
\end{align}}
for some constant $C$, by the { estimates provided in Step 3.} 
This shows that the total variation of $P^n$ on $[0,T]$ is bounded
uniformly on $n$. {On the other hand, by construction, it is easy to
see that $|P^n_t - P_t|\leq \frac{\overline P}{n}$, so that $P_t^n \to
P_t$ as $n\to \infty$ for all $t$. Then, by passing to the limit in \eqref{boundedvar}, we can conclude that $P_t$
has bounded variation on $[0,T]$.}\\

Part ii. By the arguments of the first part of the proof, the mappings  
$$
t\mapsto \int_{\overline \Omega} x \eta^m_t(dx)
$$
and, for every $n$ and $k=0,\dots,n$, 
$$
t\mapsto \int_\Omega F(k\overline P/n - x)\omega^m_t(dx) +
F_0(k\overline P/n)
$$ 
have variation bounded uniformly on $m$. Therefore, one can find a
subsequence $(m_q)_{q\geq 1}$ along which these mappings converge to
certain limits in $L^1([0,T])$. Moreover, since $\overline \Omega$ is
bounded, in view of weak convergence of measures, for any bounded continuous function $\phi:[0,T]\mapsto \mathbb R$,
$$
\lim_m \int_0^T \phi(t)\int_{\overline \Omega} x \eta^m_t(dx) =
\int_0^T \phi(t)\int_{\overline \Omega} x \eta^*_t(dx),
$$
which means that also 
$$
\int_{\overline \Omega} x \eta^{m_q}_\cdot(dx) \xrightarrow{L^1([0,T])} \int_{\overline \Omega} x \eta^*_\cdot(dx).
$$
Similar arguments show that 
$$
\int_\Omega F(k\overline P/n - x)\omega^{m_q}_\cdot(dx)
\xrightarrow{L^1([0,T])}\int_\Omega F(k\overline P/n -
x)\omega^*_\cdot(dx). 
$$
{Since the mapping $\Theta_n$ introduced at Step 4 is Lipschitz,} this
implies that the sequence of discretized prices
{$P^n_\cdot(\omega^{m_q}_t, \eta^{m_q}_\cdot)$ also converges in
$L^1([0,T])$ to the discretized price $P^n_\cdot(\omega^*_\cdot,
\eta^*_\cdot)$. On the other hand, we have seen that
$|P^n_t(\omega^*_t,\eta^*_t) - P_t(\omega^*_t,\eta^*_t)|\leq \frac{\overline
  P}{n}$. Therefore, for any $\varepsilon>0$, by taking $n\geq
\frac{3\overline P T}{\varepsilon}$} and then choosing $q$ such that 
$$
\|P^n_\cdot(\omega^{m_q}_\cdot, \eta^{m_q}_\cdot) - P^n_\cdot(
\omega^{*}_\cdot, \eta^{*}_\cdot)\|_{L^1} \leq \frac{\varepsilon}{3}, 
$$
we have that 
$$
\|P_\cdot(\omega^{m_q}_\cdot, \eta^{m_q}_\cdot) - P_\cdot(
\omega^{*}_\cdot, \eta^{*}_\cdot)\|_{L^1} \leq \varepsilon,
$$
which proves the second part of the lemma.

\end{document}